\documentclass[12pt]{article}

\usepackage{titlesec}
\usepackage{amsmath,amsthm}
\usepackage{amsfonts,xcolor,graphicx,url}
\usepackage{epstopdf,epsfig}

\usepackage[margin=2.5cm]{geometry}

\titleformat{\section}
{\normalfont\bfseries}{\thesection}{1em}{}
\titleformat{\subsection}
{\normalfont}{\thesubsection}{1em}{}

\numberwithin{equation}{section}

\newcommand{\tens}{\mathrm}
\renewcommand{\vec}{}

\theoremstyle{plain}
\newtheorem{theorem}{Theorem}[section]
\newtheorem{lemma}[theorem]{Lemma}
\newtheorem{proposition}[theorem]{Proposition}
\newtheorem{corollary}[theorem]{Corollary}
\theoremstyle{definition}
\newtheorem{definition}[theorem]{Definition}

\theoremstyle{remark}
\newtheorem{remark}{Remark}
\newtheorem{note}{Note}

\begin{document}

\title{Piecewise Flat Curvature and Ricci Flow in Three Dimensions}

% article
\author{Rory Conboye, Warner A. Miller}
\AtEndDocument{\raggedleft {\sc \ \\ Department of Physics, \\ Florida Atlantic University, \\ Boca Raton, \\ FL 33431-0991} \\ \ \\ rconboye@fau.edu}

% iop
%\author{Rory Conboye, Warner A. Miller}
%\address{Department of Physics, Florida Atlantic University, Boca Raton, FL 33431-0991}

% ipj
%\author{Rory Conboye \and Warner A. Miller}
%\address{Department of Physics, \\ Florida Atlantic University, \\ Boca Raton, \\
%FL 33431-0991}

% svj
%\author{Rory Conboye \and Warner A. Miller}
%\institute{Department of Physics, Florida Atlantic University, Boca Raton, FL 33431-0991}

%\email{rconboye@fau.edu}

% article, ijp, svj
\date{}
\maketitle

\begin{abstract}
Discrete forms of the scalar, sectional and Ricci curvatures are constructed on simplicial piecewise flat triangulations of smooth manifolds, depending directly on the simplicial structure and a choice of dual tessellation. This is done by integrating over volumes which include appropriate samplings of hinges for each type of curvature, with the integrals based on the parallel transport of vectors around hinges. Computations for triangulations of a diverse set of manifolds show these piecewise flat curvatures to converge to their smooth values. The Ricci curvature also gives a piecewise flat Ricci flow as a fractional rate of change of edge-lengths, again converging to the smooth Ricci flow for the manifolds tested.

\

% article, iop

% svj
%\keywords{

%\pacs
%\PACS{04.60.Nc, 02.40.Sf}

%\ams{}
%subclass{}
\end{abstract}

%\submitto{\CQG}

\section{Introduction}

%% Motivation
%% Piecewise flat manifolds and triangulations of smooth manifolds

Piecewise flat manifolds have long been seen as a discrete computational tool for smooth manifolds. Simplicial piecewise flat manifolds are formed by joining Euclidean $n$-simplices (line segments, triangles, tetrahedra, etc.) along their boundaries and are entirely determined by their graphs and set of edge-lengths. These are well-defined manifolds themselves and are commonly used as triangulations of smooth manifolds.

%% Combinatorial functions

The most fundamental of piecewise flat constructions use only the simplicial structure, i.e. the edge-lengths and the deficit angles around co-dimension-2 simplices, or hinges. These have combinatorial type formulations, and provide analogues of smooth manifold constructions, usually by emphasizing certain properties of their smooth counterparts.
 The Regge analogue of the Einstein equations take this form \cite{Regge}, as well as the combinatorial Ricci and scalar curvatures of Forman \cite{Forman}, Cooper $\&$ Rivin \cite{CoopRiv} and Glickenstein \cite{Glick}, and the combinatorial Yamabe and Ricci flows of Luo \cite{LuoYam}, Glickenstein \cite{GlickYam, GlickYamMax} and Chow  $\&$ Luo \cite{CombRF}.
 While the Regge equations, for example, can provide criteria for a piecewise flat manifold to be a triangulation of an Einstein manifold, and combinatorial Ricci flow may uniformize a triangulation of a smooth manifold, these constructions will not in general give values related to their smooth counterparts. Even on different triangulations of the same smooth manifold, these constructions will not necessarily be comparable.

%% Dual tessellations

In order to converge to smooth values, local scaling information is also required. This can be seen by comparing piecewise flat and smooth units of measurement, or by noticing that the simplicial structure represents integrals of smooth properties. In its most basic form this scaling information can be given by a vertex-based dual tessellation of the piecewise flat manifold. The most fundamental of constructions which converge to the smooth are essentially combinatorial in nature, but with additional scaling from a dual tessellation.
 This approach has been used in discrete exterior calculus \cite{HiraniPhD,DEC2}, in Regge calculus \cite{HambWill,Mbbp,SRF,TraceK} and in discrete differential geometry in general \cite{BobSpring}.
% A piecewise flat `simplicial Ricci flow' has also been developed by the second author along these lines \cite{SRF}, although it has recently been found to be more restrictive than first thought - giving results only for a small class of piecewise flat manifolds.
 A disadvantage of this approach is that there are many different types of dual tessellations, though the Voronoi and barycentric are by far the most common.

%% Applications for scaled combinatoric methods

While the ambiguity in choosing a dual tessellation has in part led to the popularity of the combinatorial constructions, the inclusion of such scaling information is a requirement for certain applications. Most obviously, such piecewise flat constructions can be used to approximate smooth curvatures and curvature flows. They can also be used to give insight into topological changes in smooth manifold flows, such as singularity formation \cite{MNeckPinch1,MillNand}, something which cannot be achieved with finite element or other interpolating methods. The emergence of smooth behaviour from discrete structures can also be achieved only through scaled combinatorial methods. This is becoming quite topical in physics, with many quantum gravity models suggesting that space and time may be discrete in nature \cite{RovSmo,NonAss,CausSets}, and many of these already formulated in terms of piecewise flat manifolds \cite{RegWill,RovVid}.

%% Numerical interpolation

% PF manifolds can also be used to numerically approximate smooth functions using finite element or other interpolating techniques. These have been used to successfully approximate smooth Ricci flow on different types of manifolds by Rubenstein and Sinclair \cite{RubSinc} and Garfinkle and Isenberg \cite{GarfIsen} among others. Recently, Brewin has developed a new interpolating technique for the Ricci tensor and Ricci flow in \cite{BrewinRF}. However the use of interpolation schemes reduces the direct geometric insight given by the above approach.

%% Specific structures developed here

In this paper,
 the piecewise flat scalar curvature at each vertex, and the three-dimensional sectional curvature orthogonal to each edge, are constructed by integrating over specific sets of volumes.
 These volumes are based on a choice of dual tessellation, though the specific choice is left open, and are defined to give the most appropriate sampling of hinges for each type of curvature.
 The integrals are built from the local integrated sectional curvatures of certain $2$-surface regions, which are defined by the parallel transport of vectors around the region boundaries.
 The three-dimensional Ricci curvature tangent to an edge is then given by a combination of the scalar and sectional curvatures, with the Ricci flow following directly.

The main results of the paper are highlighted in theorem \ref{thm:1.1} below, where the tessellations are assumed to be either Voronoi or barycentric in order to simplify the equations slightly compared with those appearing later in the paper.

\begin{theorem}
\label{thm:1.1}

\

\begin{enumerate}
\item For a $2$-surface region $D$ enclosing a single hinge $h$, making an angle $\theta$ with the plane orthogonal to $h$, the integrated sectional curvature over $D$ is consistent to leading order in the deficit angle $\epsilon_h$ and is given by
\begin{equation}
\int_D K^\theta \, \mathrm{d} A
 = \cos \theta \, \epsilon_h + O(\epsilon_h^2) .
 \label{1IK}
\end{equation}

\item The average scalar curvature over the Voronoi or barycentric dual volume $V_v$ at a vertex $v$ is
\begin{equation}
R_v
 := \widetilde R_{V_v}
 = \frac{1}{|V_v|}
   \sum_{h \subset \mathrm{star}(v)} |h| \epsilon_h ,
 \label{1R}
\end{equation}
with $|V_v|$ and $|h|$ representing the volume measures of $V_v$ and $h$ respectively.

\item In three-dimensions, over a volume $V_\ell$ formed by the points of the volumes dual to the vertices of an edge $\ell$ that are orthogonally related to $\ell$, the average sectional curvature orthogonal to $\ell$ is
\begin{equation}
K_\ell := \widetilde K_{V_\ell}
 = \frac{1}{|V_\ell|} \left(
     |\ell| \epsilon_\ell
   + \sum_{h} \frac{1}{2} \, |h| \, \cos^2 \theta_h \, \epsilon_h
   \right) ,
 \label{1K}
\end{equation}
with the sum taken over all other edges (hinges) $h$ intersecting $V_\ell$, and $\theta_h$ giving the angle between $h$ and $\ell$.
%   + \sum_{i} \frac{1}{2} \, |\ell_i| \, \cos^2 \theta_i \, \epsilon_i
%   \right) ,
% \label{1K}
%\end{equation}
%with the sum taken over all other edges (hinges) $\ell_i$ intersecting $V_\ell$, with $\theta_i$ giving the angle between $\ell_i$ and $\ell$, and $\epsilon_i$ the deficit angle around $\ell_i$.

\item The average Ricci curvature along an edge $\ell$ in a three-dimensional piecewise flat manifold is
\begin{equation}
\tens{Rc}_\ell
 := \widetilde{\tens{Rc}} (\hat \ell, \hat \ell)
  = \frac{1}{4} \left(R_{v_1} +  R_{v_2}\right) - K_\ell ,
 \label{1Rc}
\end{equation}
for the vertices $v_1$ and $v_2$ in the closure of $\ell$.

\item The Ricci flow of a smooth manifold $M^3$ is approximated by the edge-length equation
\begin{equation}
\frac{1}{|\ell|} \frac{d |\ell|}{d t} = - \mathrm{Rc}_\ell ,
 \label{1RF}
\end{equation}
for a piecewise flat triangulation $S^3$ of $M^3$.
\end{enumerate}
\end{theorem}

%{
%\renewcommand{\thetheorem}{\ref{thm:R}}
%\addtocounter{theorem}{-1}
%}

%% Computations

Computations for the above constructions have been carried out for triangulations of the $3$-sphere, $3$-cylinder, Gowdy model and Nil-$3$ geometry. The curvatures are shown to converge to the smooth values for triangulations of increased resolution, with errors on the order of $1 \%$ for triangulations with roughly 150 tetrahedra in the case of the latter two. The Ricci flow is computed analytically for the $3$-sphere and $3$-cylinder, showing clear convergence to the smooth flow. The Gowdy and Nil-3 Ricci flows were computed numerically, with initial results matching the smooth behaviour for short times, and this time increasing for higher resolution triangulations

%% SRF

Recently, one of us suggested a piecewise flat `simplicial Ricci flow' \cite{SRF} that yields convergence to smooth values only for geometries that have a high degree of symmetry. This method also requires specific Delaunay lattices and gives the set of edge-equations as a sparse-array matrix. The piecewise flat Ricci flow developed here can be computed much more efficiently and for a wider class of triangulations, with computations already showing convergence to smooth values for a larger class of manifolds.

%% Layout

The paper begins by giving definitions and notation for piecewise flat manifolds, their use as triangulations of smooth manifolds and the scalar, sectional and Ricci curvatures on smooth manifolds. The integrated sectional curvature around single hinges is then given in section \ref{sec:IK}, both for surfaces orthogonal to the hinge and those at an angle. Sections \ref{sec:Sca} and \ref{sec:Sec} give motivations for the choice of volumes over which to compute the scalar and sectional curvatures, with the expressions (\ref{1R}) and (\ref{1K}) proved for these volumes. Assuming these, the Ricci curvature expression in (\ref{1Rc}) is then proved in section \ref{sec:Ric}, along with the piecewise flat Ricci flow equation. Finally, the triangulation details for the test manifolds are given in section \ref{sec:Comp}, along with graphs and error results for the computations. These results give strong support to the choice of volumes over which the curvatures are constructed, and to the expressions given in theorem \ref{thm:1.1} above.

% ipj
\section*{Acknowledgements}

We acknowledge support for this research from the Air Force Research Laboratory (AFRL/RI) grant $\#$ FA8750-15-2-0047. We thank Paul Alsing, Chris Beetle and Vidit Nanda for many helpful discussions, Dan Knopf for suggesting investigation of the Nil-$3$ manifold and Christine Guenther and Mauro Carfora for introducing us to the Ricci flow of the Gowdy model.

\section{Piecewise Flat and Smooth Manifolds}
\label{sec:PF}

\subsection{Piecewise flat manifolds}

An $n$-dimensional simplicial piecewise-flat manifold is a collection of Euclidean $n$-simplices (line segments, triangles, tetrahedra), joined along their $(n-1)$-dimensional faces. Unlike other polyhedra, the properties of a flat Euclidean simplex is entirely determined by the lengths of its edges. The geometry of the whole manifold is then completely determined by the simplicial graph and the set of edge-lengths. This definition is made more precise below.

\begin{definition}[Piecewise flat manifolds]
\label{def:PLS}

\

\begin{enumerate}
\item A Euclidean $k$-simplex $\sigma^k \subset \mathbb{R}^k$ is the open interior of the  convex hull formed by $k + 1$ non-collinear points in $\mathbb{R}^k$. The convex hull itself is known as the closure of the simplex, denoted $\bar \sigma^k$.

\item A homogeneous $n$-complex is a simplicial $n$-complex where each simplex is either an $n$-simplex, or a face of an $n$-simplex.

\item A piecewise flat manifold $S^n$ is a homogeneous $n$-complex formed by Euclidean $n$-simplices, with the metric of each $\sigma^k \subset S^n$ consistent with the metrics of all $n$-simplices $\sigma^n$ where $\bar \sigma^n \supset \sigma^k$.

\item The $\mathrm{star}$ of a $k$-simplex $\sigma^k$ is the subspace of a piecewise flat manifold $S^n$ formed by the set of simplices $\sigma_i^m$ containing $\sigma^k$ in their closures, i.e. $\mathrm{star}(\sigma^k) = \{ \sigma_i^m | \bar \sigma_i^m \supset \sigma^k \}$.
\end{enumerate}
\end{definition}

The Euclidean metric of any given $n$-simplex can be smoothly developed to any neighbouring $n$-simplex, since the metric of the face between the two must be consistent with the Euclidean space of both. Deviations of $S^n$ from Euclidean space only then occur when a closed path contains one or more $(n-2)$-simplices. This deviation is measured at each $(n-2)$-simplex by comparing the sum of the dihedral angles between the $(n-1)$-simplices in its star, with a complete rotation in Euclidean space.

\begin{definition}[Deficit angles]
\label{def:DefAng}

\

\begin{enumerate}
\item The $(n-2)$-simplices in a piecewise flat manifold $S^n$ are known as \emph{hinges} and denoted by $h$.

\item In an $n$-simplex $\sigma^n$, the dihedral angle $\theta$ at a hinge $h \subset \bar \sigma^n$ is the angle between the two $(n-1)$-faces of $\bar \sigma^n$ containing $h$ in their closure, in a plane orthogonal to $h$.

\item The deficit angle $\epsilon_h$ at a hinge $h$ is given by subtracting from $2 \pi$ the sum of the dihedral angles $\theta_s$ at $h$ for each of the $n$-simplices $\sigma^n_s$ where $\bar \sigma^n_s \supset h$,
\begin{equation}\label{DefAng}
\epsilon_h := 2 \pi - \sum_{s | \sigma^n_s \supset h} \theta_s .
\end{equation}
\end{enumerate}
\end{definition}

\begin{figure}[h]
\begin{center}
\includegraphics[scale=0.25]{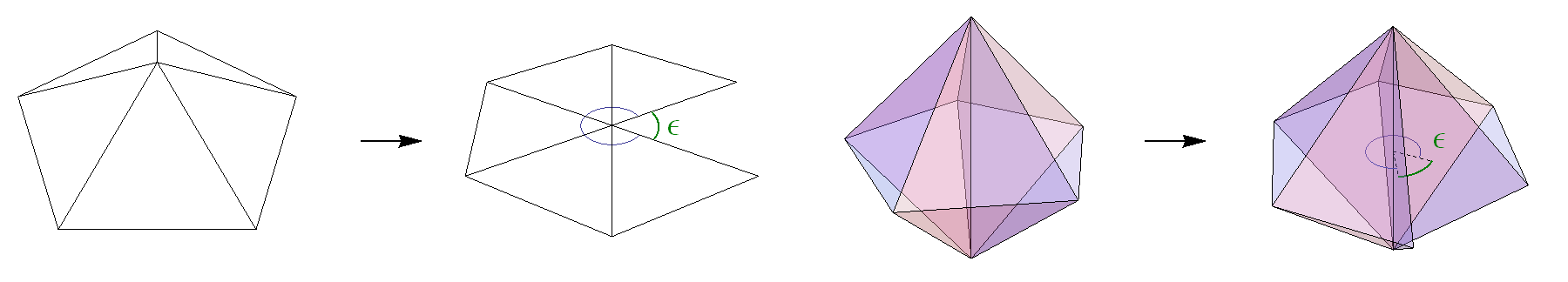}
\end{center}
\vspace{-0.5cm}
\caption{The stars of hinges in dimensions two and three, and embeddedings of their $n$-simplices in $\mathbb{R}^2$ and $\mathbb{R}^3$ respectively, with their deficit angles shown.}
\label{fig:2Dfig}
\end{figure}

The hinges of a piecewise flat manifold $S^n$ are known as \emph{conic singularities}, since in two dimensions the star of a hinge can be embedded in $\mathbb{R}^3$ as a cone shape, which is intrinsically flat everywhere but at its vertex. In higher dimensions the star of a hinge is isomorphic to the direct product of a two dimensional cone and the space of the hinge itself.

\subsection{Intrinsically approximating smooth manifolds}

It is common to approximate a smooth $n$-dimensional manifold $M^n$ by embedding it in a higher dimensional Euclidean space $\mathbb{R}^m$. Where possible, a piecewise flat manifold $S^n$ is then defined in $\mathbb{R}^m$, with either the vertices coinciding with the manifold $M^n$, the $n$-simplices of $S^n$ tangent to $M^n$, or some similar correspondence.

There is much to be gained however, by approximating smooth manifolds \emph{intrinsically}, without referring to any embedding. This can be done by forming the graph of a homogeneous simplicial complex on a smooth manifold $M^n$, and using the geodesic lengths of the edges to define a piecewise flat manifold $S^n$. This method of approximation was used at least as early as Regge \cite{Regge}, if not earlier.

\begin{definition}[Intrinsic triangulation]
\label{def:Triang}

A triangulation of a smooth manifold $M^n$ is given by a piecewise flat manifold $S^n$, defined according to the following procedure:
\begin{enumerate}
\item Construct the graph of a simplicial complex on $M^n$ by designating points for vertices, and using shortest-length geodesics between these points for the edges $\ell_M := \sigma^1_M \subset M^n$. The graph should be dense enough so that shortest-length geodesics are unique and well-defined.

\item Define a piecewise flat manifold $S^n$ using the same simplicial complex, with the lengths of the edges $\ell \subset S^n$ given by the geodesic lengths of the corresponding edges $\ell_M \subset M^n$,
\begin{equation}\label{Geodesic}
|\ell| := \int_{\ell_M} \sqrt{\tens{g}_M
 \left(\frac{d \ell_M}{d s}, \frac{d \ell_M}{d s}\right)}
 \ \mathrm{d} s ,
\end{equation}
for some parameter $s$ of the geodesic segment $\ell_M$, and $\tens{g}_M$ the metric on $M^n$.

%%  Should the metric notation above, and inner product notation later
%%  be unified to one or the other?

\item Globally rescale all of the edge-lengths of $S^n$ by a constant factor, in order to give the same global volume as $M^n$.

\item A higher density of vertices in the simplicial complex will generally give higher resolutions of $M^n$, and hence better approximations. The size of the deficit angles throughout $S^n$ give an indication of the deviation of the regions around each hinge from flat space. Hence, smaller deficit angles indicate closer approximations to the smooth manifold, relating the resolution to the curvature rather than the $n$-volume of the manifold.
\end{enumerate}
\end{definition}

%%  Finite simplexes? Compact $M^n$?

%%  Should really talk about convergence,
%%  and a more precise definition of higher resolutions
%%  (relating star of hinges to open neighbourhoods of Euclidean space?)

\begin{figure}[h]
\begin{center}
\includegraphics[scale=0.25]{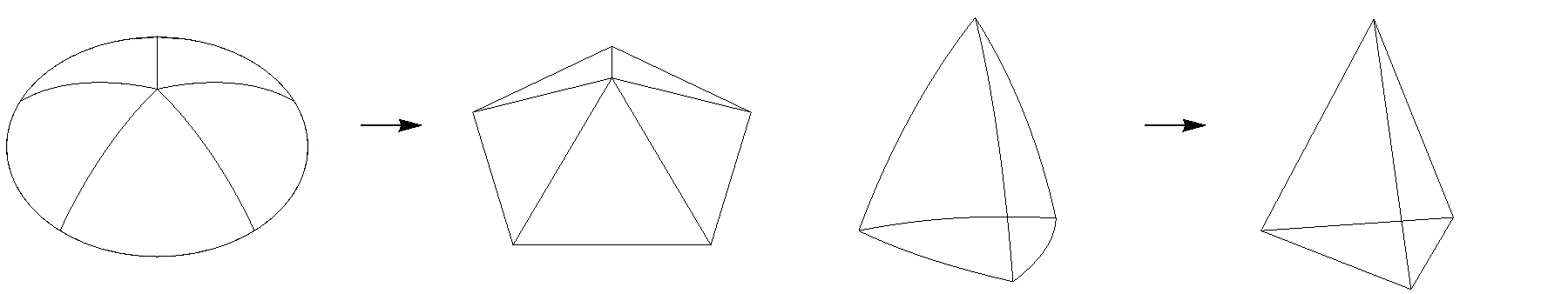}
\end{center}
\vspace{-0.5cm}
\caption{The outline of a simplicial complex on part of a $2$-sphere with its triangulation on the left, and a geodesic tetrahedron with its corresponding Euclidean tetrahedron on the right.}
 \label{fig:Approx}
\end{figure}

The set of squared edge-lengths $\{|\ell|^2\}$ can be seen as a discretization of the metric from equation (\ref{Geodesic}) above. While the simplicial complex gives the topological structure of a piecewise flat manifold $S^n$, the set of edge-lengths completely fixes the geometry of each simplex, and therefore that of $S^n$. More detail about the edge-lengths as the piecewise flat metric can be found in \cite{Sch}.

%Each edge can also be separated into two parts, one associated with each of the two vertices in its closure. The lengths of these edge-segments can also be used to represent the analogue of the metric. When all of the edge-segments connected to a given vertex have the same length, the piecewise flat manifold can be seen to come from a sphere-packing \cite{CoopRiv}. In this case, the lengths of the edge-segments associated with each vertex can be treated as a conformal factor, but the conformal class of the manifold cannot be changed without changing the structure of the simplicial complex. However this approach has been generalized to a larger class of piecewise flat manifolds by Glickenstein \cite{Glick}.

%\subsection{Deficit angles from smooth geometry}
\label{sec:DefAng}

%In two dimensions the angles in each triangle lie between two edges. This gives a direct correspondence between the deficit angle at a vertex $v$ and the deviation of neighbouring geodesics at the corresponding point $v_M$ in the smooth manifold. The deficit angle $\epsilon_v$ can therefore be seen as the integral of the Gaussian curvature over some area surrounding the point $v_M \in M^2$.

%In higher dimensions it may seem as though the deficit angles are related to a projection of the Riemann curvature tensor. However the dihedral angles are measured between co-dimension-$1$ simplices, and there is no unambiguous correspondence between a $k$-simplex $\sigma^k \subset S^n$ and a $k$-dimensional submanifold of $M^n$ for $k > 1$. The deficit angles must therefore come from a more complicated interaction between all of the geodesic segments associated with the edges in the closure of the star of a hinge, not just the parts orthogonal to it.

Since the deficit angles are computed orthogonally to each hinge, they can in some sense be seen as integrals of the projection of the Riemann curvature orthogonal to the hinges. Even in three dimensions this does not hold exactly however, since the orientation of a plane orthogonal to a geodesic segment $\ell_M$ will not in general be consistent along its length. The measure of the dihedral angle at an edge $\ell$ in a tetrahedron also depends purely on the \emph{lengths} of the geodesic segments in the closure of the corresponding smooth tetrahedron, and in no way to their relative orientations.

\begin{remark}
Both Friedberg \& Lee \cite{FriedLee} and Christiansen \cite{Christ} consider a one parameter family of smooth manifolds $M^n_t$ which limit to a given piecewise flat manifold $S^n$. This is in contrast to a family of piecewise flat triangulations of a given \emph{smooth} manifold. In their case, close to each hinge $h_M$, the Riemann curvature of $M^n_t$ approaches its projection orthogonal to $h_M$, giving the deficit angle $\epsilon_h$ when integrated over an appropriate surface orthogonal to $h_M$.	
\end{remark}

\subsection{Smooth curvature}
\label{sec:SmoothCurv}

On a smooth manifold $M^n$, the entire intrinsic curvature of the manifold is given by the Riemann curvature tensor. This tensor gives the change in a vector after it has been parallel transported around an infinitesimal closed path,
\begin{equation}	
\tens{Rm} (\vec{u}, \vec{v}) \vec{w}
 := \nabla_{\vec{u}} \nabla_{\vec{v}} \vec{w}
 - \nabla_{\vec{v}} \nabla_{\vec{u}} \vec{w} ,
\label{Riem}
\end{equation}
for orthogonal vector fields $\vec{u}$ and $\vec{v}$, with $\nabla$ representing the covariant derivative for the Levi-Civita connection on $M^n$. The sectional curvature gives a restriction of the Riemann curvature to a particular infinitesimal $2$-plane. Specifying two orthogonal unit vectors $\vec{u}$ and $\vec{v}$, the sectional curvature for the plane spanned by $\vec{u}$ and $\vec{v}$ is defined as
\begin{equation}	
K (\vec{u}, \vec{v})
 := \left< \tens{Rm} ( \vec{v}, \vec{u} ) \vec{u}, \vec{v} \right>
 = \left< \nabla_{\vec{v}} \nabla_{\vec{u}} \vec{u}
 - \nabla_{\vec{u}} \nabla_{\vec{v}} \vec{u}, \vec{v} \right> .
\label{Sec}
\end{equation}
with $\left< \cdot, \cdot \right>$ representing the inner product associated with the metric $\tens{g}_M$ on $M^n$. The Ricci and scalar curvatures are given by contractions of the Riemann tensor. For any orthonormal basis $\{ \vec{e_1}, \vec{e_2}, ... , \vec{e_n} \}$ in a neighbourhood of $p \in M^n$, the Ricci tensor at $p$ is
\begin{equation}
\tens{Rc}(\vec{e_i}, \vec{e_j})
 := \sum_{k = 1}^n \left< \tens{Rm}(\vec{e_k}, \vec{e_i}) \vec{e_j},
   \vec{e_k} \right> ,
\label{Ric}
\end{equation}
with the scalar curvature given by a further contraction of the Ricci tensor,
\begin{equation}
R :=
 \sum_i^n \tens{Rc}(\vec{e_i}, \vec{e_i})
 = \sum_i^n \sum_j^n
 \left< \tens{Rm} (\vec{e_j}, \vec{e_i}) \vec{e_i}, \vec{e_j} \right>
 = 2 \sum_{i < j}
 K(\vec{e_i}, \vec{e_j}) .
\label{Scal}
\end{equation}
The last part above comes directly from the definition of the sectional curvature (\ref{Sec}), showing the scalar curvature to be twice the sum of the sectional curvatures for a complete set of orthogonal $2$-planes.

% For a dense enough triangulation $S^n$ of a smooth manifold $M^n$, the average of the integral of a curvature over an appropriate choice of $n$-volume in $S^n$ will approximate the average value of the curvature over a corresponding $n$-volume in the smooth manifold.

Since a triangulation of a smooth manifold can be seen as a discretization, it is the average of curvature components over a discrete set of volumes that should give a correlation between the piecewise flat and smooth manifolds. The choice of this set of volumes will depend on the type of curvature, but should contain an appropriate sampling of hinges in order to compare properly with the smooth curvature. The average of each curvature component over a given volume can then be found from their integrals over that volume.

Sectional curvature integrals over almost-flat $2$-surface regions are well-defined in terms of the parallel transport of vectors around the region boundaries, a straight forward procedure within a piecewise flat manifold.
 Conveniently, the curvatures above can all be defined in terms of the sectional curvatures over certain collections of $2$-planes.
 For volumes which can be foliated into parallel almost-flat $2$-surfaces, the integrals of these curvatures can be given in terms of the integrated sectional curvatures over these surfaces.
 Such foliations and $2$-surface sectional curvature integrals will be covered in the next section.

\section{Local Integrated Sectional Curvature}
\label{sec:IK}

\subsection{Smooth integrated sectional curvature}

In a smooth manifold, the integral of the sectional curvature over a region of an almost-planar $2$-surface can be defined in terms of the parallel transport of a vector around the boundary of the region.

\begin{lemma}[Smooth integrated sectional curvature]
\label{lem:IntSecM}
Take a region $D$ of the $2$-surface $P(\vec{u},\vec{v})$ spanned by the orthonormal vector fields $\vec{u}$ and $\vec{v}$, where $P(\vec{u},\vec{v})$ is almost a flat $2$-plane within $D$. The integral of the sectional curvature over $D$ can be approximated by the $\vec{v}$ component of the parallel transport of $\vec{u}$ around $\partial D$, the boundary of $D$,
\begin{equation}
\int_D K(\vec{u},\vec{v}) \, \mathrm{d} A
 \ \simeq \ \left< \oint_{\partial D}
 \mathrm{d} \vec{s}^a \, \nabla^{P}_a \vec{u}, \, \vec{v} \right> ,
 \label{IntKM}
\end{equation}
with $s$ representing some parameter of the boundary $\partial D$, and $\nabla^{P}$ the covariant derivative of the Levi-Civita connection intrinsic to $P(\vec{u},\vec{v})$.
\end{lemma}

The index notation is used to give components in some orthonormal basis, and the Einstein summation convention is assumed for indices appearing both above and below.

\begin{proof}
The values of the sectional curvature $K(\vec{u}, \vec{v})$, for the tangent planes formed by $\vec{u}$ and $\vec{v}$ at each point, can be integrated over $D$ giving the expression
\begin{equation}
\int_D K(\vec{u},\vec{v}) \, \mathrm{d} A
 = \int_D \left<
  \tens{Rm}(\vec{v}, \vec{u})\vec{u}, \vec{v} \right> \mathrm{d} A
 \ \simeq \ \left< \int_D
   (\mathrm{d} \vec{v} \wedge
   \mathrm{d} \vec{u})^{a b} \
   \tens{Rm}_{a b} \vec{u}, \,
   \vec{v} \right>_p ,
\label{IntSecMproof}
\end{equation}
%%  Need to make clear what exactly is meant be inner product of integral,
%%  Can we say anything about sum of inner products
%%                            versus inner product of sum?
%%  Require proper definition of region $D$ that allows relation above.
with the approximate equivalence holding for any $p \in D$, as long as the inner product on the right is invariant to the choices of $\vec{u}$, $\vec{v}$ and the point $p$, to some degree of approximation. This will be the case where the inner product is close to Euclidean, i.e. when the region $D$ is almost a flat $2$-plane. Using Stoke's theorem, the integral of the Riemann tensor can be reduced to a contour integral around $\partial D$,
\begin{equation}
\int_D (\mathrm{d} \vec{v} \wedge \mathrm{d} \vec{u})^{a b} \
   \tens{Rm}_{a b} \vec{u}
 \ \simeq \ \oint_{\partial D}
 \mathrm{d} \vec{s}^a \, \nabla^{P}_a \vec{u} ,
\end{equation}
with the required result following from substitution into (\ref{IntSecMproof}) above.

\end{proof}

Over suitable $2$-surfaces in a piecewise flat manifold $S^n$ this relation will be used as a local definition of the integrated sectional curvature.
 This definition is considered \emph{too} local for comparison with the smooth sectional curvature, for example a region which contains no hinges will be Euclidean, giving a zero sectional curvature integral.
 However it will be used in sections \ref{sec:Sca} and \ref{sec:Sec} to construct appropriate \emph{volume} integrals of the scalar and sectional curvatures, which should be comparable to their corresponding smooth values.

\subsection{Two dimensions and hinge-orthogonal surfaces}

In two dimensions, the parallel transport of any vector around a closed path enclosing a single vertex $v$ is rotated by exactly the deficit angle $\epsilon_v$. A region $D$ enclosing $v$ and no other vertices will be part of a $2$-surface that is locally isomorphic to $\mathbb{R}^2$ everywhere but at $v$. The integral of the sectional curvature over $D$ is then given by the inner product of the parallel transported vector, with a vector which was initially orthogonal to it. This can easily be seen to give
\begin{equation}
\int_{D} \, ^{(2)}K \ \mathrm{d} A
 = \sin \epsilon_{v}
 = \epsilon_{v} + O(\epsilon_h^3) ,
\label{IK2D}
\end{equation}
with the first term alone giving a good approximation for small deficit angles.

This result can easily be extended to higher dimensions for surfaces which are orthogonal to a hinge $h$.

\begin{lemma}[Hinge-orthogonal planes and sectional curvature]
\label{lem:PiPerp}
\label{lem:IKPerp}
In a piecewise flat manifold $S^n$, the star of a hinge $h$ can be foliated into parallel $2$-surfaces $P^\perp_h$ orthogonal to $h$, which are isomorphic to $\mathbb{R}^2$ everywhere but at their intersection with $h$. The integrated sectional curvature over any area $D \subset P^\perp_h$ enclosing $h$ is then
\begin{equation}
{\cal K}^\perp_h
 := \int_D K (P^\perp_h) \ \mathrm{d} A
 = \sin \epsilon_h
 = \epsilon_{h} + O(\epsilon_h^3) .
 \label{IKPerp}
\end{equation}
% with the first term giving a good approximation for small deficit angles, which should be given by a dense-enough triangulation.
\end{lemma}

\begin{proof}
Each $n$-simplex $\sigma^n_h$ in the star of $h$ can be foliated into flat $2$-planes orthogonal to $h$, since these simplices are Euclidean and contain $h$ in their closure. Orthogonality is consistent across any pair of simplices, so the $2$-planes orthogonal to $h$ at each point $p \in h$ must coincide across all of the $n$-simplices $\sigma^n_h$. Hence, for each $p \in h$ there is a unique $2$-surface orthogonal to $h$ across its star, which is isomorphic to $\mathbb{R}^2$ everywhere but at $p \in h$. These surfaces must then form a foliation of the star of $h$.

Since deficit angles are defined orthogonally to each hinge, each plane $P_h^\perp$ acts like the star of a hinge in $2$-dimensions. The parallel transport of a unit vector $\vec{u} \in P_h^\perp$ around $h$ remains in $P_h^\perp$ and is rotated by $\epsilon_h$ exactly. The inner product of this with a unit vector $\vec{v} \in P^\perp_h$ orthogonal to $\vec{u}$ then gives $\sin \epsilon_h$, independent of the point at which the inner product is taken and the choice of vectors $\vec{u}$ and $\vec{v}$.

\end{proof}

\subsection{Non-hinge-orthogonal surfaces}

In dimensions greater than two, there are other possible orientations of surfaces with respect to a hinge. The star of a hinge cannot be foliated into \emph{planes} if they are not orthogonal to $h$ however. Instead, a generalization of almost-planar $2$-surfaces is given below, with the integral of the sectional curvature over such surfaces then shown to be consistent to leading order in the deficit angle.

\begin{lemma}[Non-hinge-orthogonal $2$-surfaces]
\label{lem:Pitheta}
The star of a hinge $h$, with a small deficit angle $\epsilon_h$, can be foliated into parallel almost-planar $2$-surfaces $P^\theta_h$ which make an angle of $\theta$ with the surfaces $P^\perp_h$. These $2$-surfaces are locally isomorphic to $\mathbb{R}^2$ everywhere but at their intersection point with $h$, and along a single radial line $k$ from $h$.

The line $k$ is directly opposite from the radial line which makes an angle $\theta$ with $P^\perp_h$. Within the star of $h$, the surface $P^\theta_h$ forms a corner along $k$ which deviates from $\pi$ by a maximum of $\epsilon_h$, attained as $\theta \rightarrow \pi/2$.
\end{lemma}

\begin{figure}[h]
\begin{center}
\includegraphics[scale=0.25]{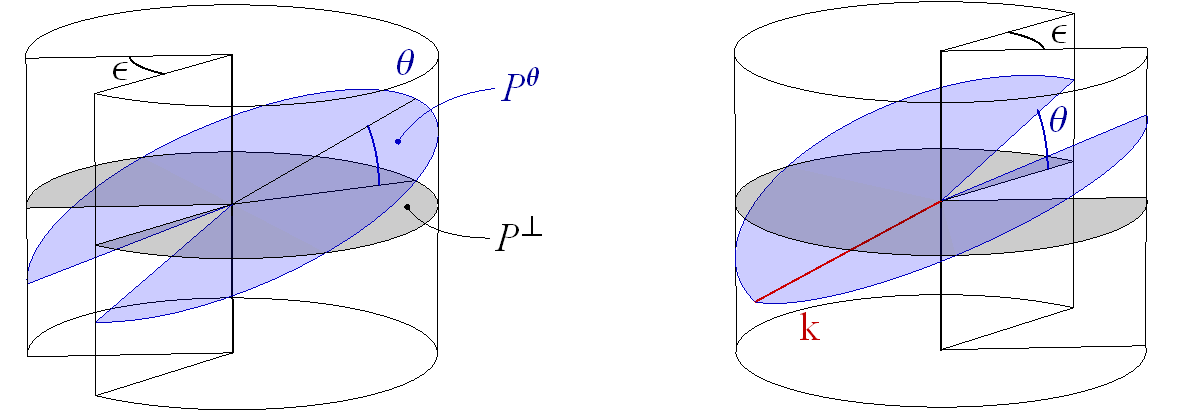}
\end{center}
\vspace{-0.5cm}
\caption{For a hinge in a three dimensional piecewise flat manifold, the diagrams give the embedding of a cylindrical subspace of $\mathrm{star}(h)$ into $\mathbb{R}^3$ showing $P_h^\perp$ and $P^\theta_h$. The deficit angle is cut out at $k$ in the first image and on the opposite side for the second, showing the corner along $k$.}
\label{fig:Cyl}
\end{figure}

\begin{proof}
Within each $n$-simplex in the star of a hinge $h$, $2$-planes at an angle $\theta$ to $P^\perp_h$ are well-defined with complete freedom in the direction of their gradient with respect to $P^\perp_h$. As with any vector, the direction of the gradient in $P^\perp_h$ can be consistently defined across the boundaries between neighbouring $n$-simplices, but gets rotated by the deficit angle $\epsilon_h$ when a hinge $h$ is enclosed. As a result, a $2$-plane at an angle $\theta$ to the orthogonal plane $P^\perp_h$ cannot be consistently defined throughout the star of $h$.

Taking a line passing through the point $p \in h$ at an angle $\theta$ to the surface $P_h^\perp (p)$, a $2$-plane containing this line can be defined locally so that it makes and angle $\theta$ with $P_h^\perp$ everywhere. This $2$-plane can be developed around both sides of $h$, with the radial lines from $p$ having a slope of $\cos \phi \, \tan \theta$, where the angle $\phi$ is measured in $P_h^\perp$ from the initial line. This can be seen in the first diagram in figure \ref{fig:thCalc}. Due to the deficit angle of the hinge, the angles $\phi = \pm (\pi - \epsilon/2)$ represent the same part of $P_h^\perp$. The radial lines of this surface from each side of $h$ make the same angle $\alpha$ with $P_h^\perp$ such that
\begin{equation}
\tan \alpha
 = \cos \pm (\pi - \epsilon/2) \, \tan \theta
 = \cos \epsilon/2 \, \tan \theta ,
\end{equation}
and since both lines also contain the same point $p \in h$, they must be equivalent. This line will be denoted $k$, and shows that the surface closes in the star of $h$. These surfaces, denoted $P_h^\theta$, exist for each $p \in h$ and with the lines $k$ coinciding with respect to $P_h^\perp$, they are parallel and form a foliation of $\mathrm{star}(h)$.

\begin{figure}[h]
\begin{center}
\includegraphics[scale=0.25]{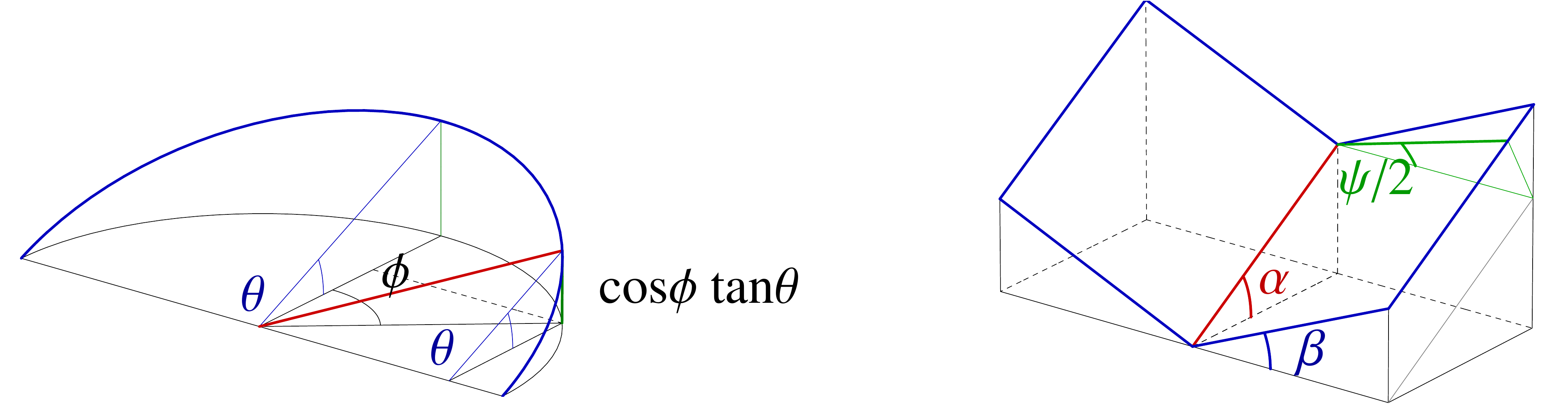}
\end{center}
\vspace{-0.5cm}
\caption{The first diagrams shows the slope of the radial lines of $P_h^\theta$ with respect to $P_h^\perp$ using part of a unit circle in the later. The second shows the computation of the angle along $k$.}
\label{fig:thCalc}
\end{figure}

The slope of the surface $P_h^\theta$ with respect to $P_h^\perp$, in a direction orthogonal to the radial direction in $P_h^\perp$, is not the same from both sides of $k$ however. The slopes from each side can be defined in terms of the angle $\beta$, which is given by the derivative of the slopes of the radial lines in $P_h^\theta$ with respect to $\phi$,
\begin{equation}
\tan \beta
 = \left. \frac{\partial}{\partial \phi}
   \cos \phi \, \tan \theta \right|_{\phi = \pm (\pi - \epsilon/2)}
 = \pm \sin \epsilon/2 \, \tan \theta .
\end{equation}
The surface $P_h^\theta$ must therefore form a \emph{corner} along $k$ within the space of $\mathrm{star}(h)$. This can be seen in the right-hand diagrams of both figures \ref{fig:Cyl} and \ref{fig:thCalc}. It can also be seen from figure \ref{fig:thCalc} that this corner deviates from $\pi$ by an angle $\psi$ satisfying the equation
\begin{equation}
\tan \psi/2
 = \cos \alpha \, \tan \beta = \frac{\sin \epsilon/2 \, \tan \theta}
        {\sqrt{1 + \cos^2 \epsilon/2 \, \tan^2 \theta}} .
\label{psi}
\end{equation}
The angle $\psi$ can easily be seen to vanish at $\theta = 0$, where $P_h^\theta$ coincides with $P_h^\perp$. Its maximum absolute value then occurs as $\theta \rightarrow \pm \pi/2$, with $\psi \rightarrow \pm \epsilon$.

\end{proof}

\begin{theorem}[Non-hinge-orthogonal integrated sectional curvature]
\label{thm:IKtheta}
Within the star of a hinge $h$, the integrated sectional curvature over any area $D \subset P^\theta_h$ intersecting $h$ is consistent to leading order in $\epsilon_h$, and is given by
\begin{equation}
{\cal K}^\theta_h
 := \int_D K (P^\theta_h) \, \mathrm{d} A
 = \cos \theta \, \epsilon_h + O(\epsilon_h^2) .
 \label{IKAngPlane}
\end{equation}
\end{theorem}

\begin{proof}
For parallel transport within $P_h^\theta$, first consider the parallel transport of a set of basis vectors for $P_h^\theta$ such that one is parallel to $k$ and the other orthogonal to $k$ form a given side. In order to parallel transport the vectors within the surface $P^\theta_h$, the vector orthogonal to $k$ will have to be rotated in the ambient Euclidean space by an angle of $\psi$ at $k$, while the vector parallel with $k$ will remain unaffected by the corner. The remainder of the parallel transports can be performed within the Euclidean space of each $n$-simplex, since the surface $P^\theta_h$ is locally a Euclidean $2$-plane away from $k$.

In three dimensions, a cylindrical type coordinate system $(\rho, \phi, z)$ can be defined in the star of $h$, with $z$ parameterizing $h$, $\rho$ giving the radial distance from $h$ and $\phi$ periodic in $2 \pi - \epsilon_h$. The unit vectors $\vec{a}, \vec{b} \in P^\theta_h$, with $\vec{a}$ parallel to $k$ and $\vec{b}$ orthogonal to it, are given in this coordinate basis as
\begin{equation}
\vec{a} := (\cos \alpha, \, 0, \, - \sin \alpha), \qquad
\vec{b} := (
 - \sin \alpha \, \sin \psi/2, \,
   \cos \psi/2, \,
 - \cos \alpha \, \sin \psi/2
),
\end{equation}
where $\alpha = \arctan (\cos \epsilon/2 \, \tan \theta)$ is the angle between $k$ and the $\rho$-direction. Transporting $\vec{b}$ parallel to $P^\theta_h$ through $k$ gives the vector
\begin{equation}
T_k (\vec{b}) = (
 \sin \alpha \, \sin \psi/2, \,
 \cos \psi/2, \,
 \cos \alpha \, \sin \psi/2),
\end{equation}
rotated by $\psi$ in the plane orthogonal to $k$. The parallel transport of $\vec{a}$ around $h$, and $\vec{b}$ the remainder of the way around $h$, leaves the $z$-components unchanged, rotating the $\rho \phi$-plane by $\epsilon_h$ in the same direction as the parallel transport. This gives the transported vectors as
\begin{eqnarray}
T_h (\vec{a})
 &=& (
   \cos \alpha \, \cos \epsilon_h, \,
   \cos \alpha \, \sin \epsilon_h, \,
 - \sin \alpha
 ),	\nonumber \\
T_{h \circ k} (\vec{b})
 &=& (
   \sin \alpha \, \sin \psi/2 \, \cos \epsilon_h
 - \cos \psi/2 \, \sin \epsilon_h, \,
	\nonumber \\ && \quad
   \sin \alpha \, \sin \psi/2 \, \sin \epsilon_h
 + \cos \psi/2 \, \cos \epsilon_h, \,
   \cos \alpha \, \sin \psi/2
 ).
\end{eqnarray}
The inner products of the parallel transported vectors with the original vectors are
\begin{eqnarray}
<T_h (\vec{a}), \vec{a}>
 &=& \cos^2 \alpha \, \cos \epsilon_h + \sin^2 \alpha
	\nonumber \\
<T_h (\vec{a}), \vec{b}>
 &=&
 - \cos \alpha \, \sin \alpha \, \sin \psi/2 \, \cos \epsilon_h
 + \cos \alpha \, \cos \psi/2 \, \sin \epsilon_h
 + \cos \alpha \, \sin \alpha \, \sin \psi/2
	\nonumber \\
<T_{h \circ k} (\vec{b}), \vec{a}>
 &=&
   \cos \alpha \, \sin \alpha \, \sin \psi/2 \, \cos \epsilon_h
 - \cos \alpha \, \cos \psi/2 \, \sin \epsilon_h
 - \cos \alpha \, \sin \alpha \, \sin \psi/2
	\nonumber \\
<T_{h \circ k} (\vec{b}), \vec{b}>
 &=&
 - \sin^2 \alpha \, \sin^2 \psi/2 \, \cos \epsilon_h
 - \cos^2 \alpha \, \sin^2 \psi/2
	\nonumber \\ && \quad
 + 2 \sin \alpha \, \cos \psi/2 \, \sin \psi/2 \, \sin \epsilon_h
 + \cos^2 \psi/2 \, \cos \epsilon_h \,
.
\end{eqnarray}

Any vector within $D$ can now be decomposed into its $\vec{a}$ and $\vec{b}$ components. An orthonormal set of vectors $\vec{u}, \vec{v} \in P_h^\theta$ can be given for any choice of angle $\beta$ by
\begin{equation}
\vec{u} = \cos \beta \, \vec{a} + \sin \beta \, \vec{b}, \qquad
\vec{v} = - \sin \beta \, \vec{a} + \cos \beta \, \vec{b}.
\end{equation}
The parallel transport of $\vec{u}$ around the boundary of $D$ is given by the same coefficients as above, with the basis vectors $T_h (\vec{a})$ and $T_{h \circ k} (\vec{b})$. The inner product of the parallel transport of $\vec{u}$ with $\vec{v}$ is then
\begin{eqnarray}
<T(\vec{u}), \vec{v}> &=&
 - \cos \beta \, \sin \beta \,
 <T_h (\vec{a}), \vec{a}>
 + \cos^2 \beta \,
 <T_h (\vec{a}), \vec{b}>
	\nonumber \\ && \quad
 - \sin^2 \beta \,
 <T_{h \circ k} (\vec{b}), \vec{a}>
 + \cos \beta \, \sin \beta \,
 <T_{h \circ k} (\vec{b}), \vec{b}>
	\nonumber \\  &=&
 - \cos \alpha \, \sin \alpha \, \sin \psi/2 \, \cos \epsilon_h
 + \cos \alpha \, \cos \psi/2 \, \sin \epsilon_h
 + \cos \alpha \, \sin \alpha \, \sin \psi/2
	\nonumber \\ && \quad
 - \cos \beta \, \sin \beta \, (
   \cos^2 \alpha \, \cos \epsilon_h
 + \sin^2 \alpha \, \sin^2 \psi/2 \, \cos \epsilon_h
 - \cos^2 \psi/2 \, \cos \epsilon_h
	\nonumber \\ && \quad
 - 2 \sin \alpha \, \cos \psi/2 \, \sin \psi/2 \, \sin \epsilon_h
 + \sin^2 \alpha
 + \cos^2 \alpha \, \sin^2 \psi/2
   )
%	\nonumber \\  &=&
%   \cos \alpha \, \epsilon_h + O(\epsilon_h^3)
% - \cos \beta \, \sin \beta \, (
%	\nonumber \\ && \quad
%   \cos^2 \alpha
% - \cos^2 \alpha \, \epsilon_h^2/2
% + \sin^2 \alpha \, \psi^2/4
% - 1
% + \psi^2/4
% + \psi^2/4
% + \epsilon_h^2/2
%	\nonumber \\ && \quad
% -   \sin \alpha \, \psi \, \epsilon_h
% + \sin^2 \alpha
% + \cos^2 \alpha \, \psi^2/4
% + O(\epsilon_h^3)
%   )
%	\nonumber \\  &=&
%   \cos \alpha \, \epsilon_h + O(\epsilon_h^3)
%	\nonumber \\ && \quad
% - \cos \beta \, \sin \beta \, (
% - \cos^2 \alpha \, \epsilon_h^2/2
% - \sin \alpha \, \psi \, \epsilon_h
% + \epsilon_h^2/2
% + 3 \psi^2/4
% + O(\epsilon_h^3)
%   )
	\nonumber \\  &=&
   \cos \alpha \, \epsilon_h
 - \cos \beta \, \sin \beta \, O(\epsilon_h^2) \, .
\end{eqnarray}
Terms of order $\psi$ are also considered to be of order $\epsilon_h$, since the maximum absolute value of $\psi$ is $\epsilon_h$. If the linear terms alone are taken, then the inner product is independent of the choice of vector parallel transported. The integral of the sectional curvature can therefore be consistently defined to first order in $\epsilon_h$. From (\ref{psi}) it can also be seen that $\cos \alpha = \cos \theta$ to first order in $\epsilon_h$.

Naively, the parallel transport of a unit vector $\vec{u} \in P^\theta_h$ can be taken with respect to the space $\mathrm{star}(h)$ rather than $P^\theta_h$. This essentially removes the contribution from the corner along $k$ in the computation above. The inner product of the transported vector $\vec{u}$ with $\vec{v}$ is then
\begin{equation}
<T_{\mathrm{star}(h)}(\vec{u}), \vec{v}>
 = \cos \theta \, \epsilon_h
 - \cos \beta \, \sin \beta (1 - \cos^2 \theta) \frac{\epsilon_h^2}{2}
 + O(\epsilon_h^3) ,
\end{equation}
which is still consistent to first order in the deficit angle, and gives the same result for this level of approximation.

This can easily be extended to higher dimensions, where the $z$-coordinate is chosen to represent the linear subspace of $h$ which makes the smallest angle with $P^\theta_h$. The span of any orthogonal complement in the space $h$ can then be factored out, giving the same computations and results as above.

\end{proof}

%\begin{figure}[h]
%\begin{center}
%\includegraphics[scale=0.5]{}
%\end{center}
%\caption{The parallel transport of a vector $\vec{u}$ around $P^\theta_h$, and its inner product with the vector $\vec{v} \perp \vec{u}$.}
%\label{fig:IK}
%\end{figure}

As required, the expression for the integrated curvature vanishes for any surface which intersects $h$ at more than a single point, i.e. for $\theta = \pi/2$, and gives the deficit angle itself for $\theta = 0$.

\section{Scalar Curvature}
\label{sec:Sca}

\subsection{Two dimensional surfaces}

On a two-dimensional smooth surface $M^2$ the tangent space at each point is itself a $2$-plane, so the scalar curvature is exactly twice the sectional curvature by (\ref{Scal}), and gives the entire Riemann curvature at each point. On a piecewise flat surface $S^2$, for a region $D$ enclosing a single vertex $v$, the integrated sectional curvature is $\sin \epsilon_v$ from (\ref{IK2D}) giving the integrated scalar curvature over $D$ as $2 \sin \epsilon_v$. If the region $D \subset S^2$ encloses a number of vertices $v_i$, the integral of the scalar curvature over $D$ is the sum of the integrals for each vertex
\begin{equation}
\int_{D} \, ^{(2)}R \ \mathrm{d} A
 = 2 \sum_i \int_{D_i} K \ \mathrm{d} A
 = 2 \sum_i \sin \epsilon_{v_i}
 = 2 \sum_i \epsilon_{v_i} + O(\epsilon^3) .
\label{Ksumh}
\end{equation}
The higher order terms can then be truncated for triangulations of a high-enough resolution.

For a closed piecewise flat surface $S^2$, a combinatorial form of the Gauss-Bonnet theorem shows that the sum of all deficit angles is equivalent to $2 \pi$ times the Euler characteristic of the topology of the surface. When $S^2$ is a triangulation of a smooth surface $M^2$, the smooth Gauss-Bonnet theorem for $M^2$ implies
\begin{equation}
\int_{M^2} \, ^{(2)}R \ \mathrm{d} A \equiv 2 \sum_{h \in S^n} \epsilon_h .
\end{equation}
This shows the deficit angles to represent integrals of the scalar curvature. Decomposing the total area into local regions $D$ over which to give average curvatures is ambiguous, however it seems natural to tessellate $S^2$ into areas $A_v$, each enclosing a single vertex $v$. The piecewise flat scalar curvature at each vertex $v$ can then be defined as the average of the integrated curvature over $A_v$,
\begin{equation}
^{(2)}R_v := 2 \frac{\epsilon_v}{A_v} .
\label{R2D}
\end{equation}
With an appropriate distribution of the total area of $S^2$ over each $A_v$, related to the specifics of the triangulation in some way, $^{(2)}R_v$ can be compared with the scalar curvature at the point $v_M$ corresponding to $v$ on the smooth manifold $M^2$. Specific decompositions will not be discussed here, but appendix \ref{sec:Duals} gives details about the two most common decompositions, the Voronoi and barycentric.

\subsection{Single hinges in higher dimensions}

In higher dimensions, there is more than a single plane orientation in the neighbourhood of each point, with the the scalar curvature formed from the sectional curvatures of a set of orthogonal planes. Within the star of a hinge $h$, almost-planar $2$-surfaces orthogonal to $h$ have an integrated sectional curvature given by the deficit angle $\epsilon_h$, however the integrated sectional curvatures of surfaces orthogonal to this must vanish. The $n$-volume integral of the scalar curvature over the star of $h$ is then given in very simple terms.

\begin{lemma}[Integral of scalar curvature over star of hinge]
\label{lem:IntRh}
The integral of the scalar curvature over the $n$-volume of the star of a hinge $h$ is
\begin{equation}
\label{IRh}
\int_{\mathrm{star}(h)} R \ \mathrm{d} V^n
 = 2 |h| \sin \epsilon_h
 = 2 |h| \epsilon_h + O(\epsilon_h^3) .
\end{equation}
\end{lemma}

\begin{proof}
The integral of the scalar curvature over the star of $h$ can be given in terms of the integrals of sectional curvatures for a complete orthogonal set of almost-planar $2$-surfaces $P_i$. Using (\ref{Scal}),
\begin{equation}
\int_{\mathrm{star}(h)} R \ \mathrm{d} V^n
 = \int_{\mathrm{star}(h)} \left[
   2 \sum_{i} K(P_i)
   \right] \mathrm{d} V^n
 = 2 \sum_{i}
   \int_{\mathrm{star}(h)} K(P_i) \ \mathrm{d} V^n ,
\end{equation}
with the integral and sum commuting as long as the surfaces $P_i$ are consistent throughout the star of $h$. Choosing one of the surfaces $P_i$ at each point so that it is orthogonal to $h$, the complimentary span of surfaces cannot be consistently separated into orthogonal surfaces, however the parallel transport over any path will be parallel to $h$ and therefore cannot enclose it. The only non-vanishing sectional curvature must then come from the surfaces $P_h^\perp$. The equation above then reduces to
\begin{equation}
\int_{\mathrm{star}(h)} R \ \mathrm{d} V^n
 = 2 \int_{h} {\cal K}_h^\perp \ \mathrm{d} V^{(n-2)}
 = 2 |h| {\cal K}_h^\perp
 = 2 |h| \sin \epsilon_h ,
\label{IntRh3}
\end{equation}
since the value of ${\cal K}_h^\perp$ is independent of the point at which the surface $P_h^\perp$ intersects $h$.

\end{proof}

\begin{corollary}[Integral of scalar curvature over subspace of hinge star]
\label{cor:IntRhD}
The integral of the scalar curvature over any $n$-dimensional region $D \subset \mathrm{star}(h)$ is
\begin{equation}
\label{IRhD}
\int_{D} R \ \mathrm{d} V^n
 = 2 |h_{|D}| \sin \epsilon_h
 = 2 |h_{|D}| \epsilon_h + O(\epsilon_h^3) ,
\end{equation}
with $h_{|D}$ representing the part of $h$ that is contained within $D$, i.e. $h_{|D} = h \cap D$.
\end{corollary}

\begin{proof}
The proof follows the same argument as lemma \ref{lem:IntRh} above, with the integral in the second term of (\ref{IntRh3}) carried out over $h_{|D}$ instead of $h$.

\end{proof}

If a compact manifold is decomposed into $n$-volumes $D_h$, with each containing a complete hinge $h$ and no part of any other, then the integral of the scalar curvature is just the sum of the integrated curvatures over each $n$-volume $D_h$. Using corollary \ref{cor:IntRhD} above,
\begin{equation}
\int_{S^n} R \ \mathrm{d} V^n
 = 2 \sum_h |h| \epsilon_h ,
\label{Regge}
\end{equation}
which agrees exactly with the Regge action \cite{Regge}. This has led to the hypothesis that it is the average scalar curvature over each of these $n$-volumes $D_h$ that should be compared with the continuum \cite{HambWill,PiranWill,HilbertAction}. A number of constructions have been suggested for the regions $D_h$, the most common of which are the $n$-polytopes formed by the vertices in the closure of $h$, and either the circumcenters or barycenters of simplices in the star of $h$.

However, in the proof of lemma \ref{lem:IntRh} the only contribution to the integral comes from the deficit angle at $h$. As discussed at the end of section \ref{sec:DefAng}, this deficit angle is mostly related to the projection of the Riemann curvature orthogonal to $h$, and should therefore not contain enough information about the rest of the Riemann tensor to give the full scalar curvature. Computations in section \ref{sec:Comp} also show the deficit angles to be influenced by the orientation of hinges. This all implies that a single hinge is not enough to compare with the scalar curvature of a smooth manifold, and that a larger sampling of hinges is instead required.

\subsection{Vertex-based scalar curvature}

The hinges in the star of a vertex give a discretization of a complete span of hinge orientations. Vertex based tessellations of piecewise flat manifolds therefore provide a natural setting for the computation of the scalar curvature. As with the the $2$-dimensional case, there are a number of different methods for decomposing a piecewise flat manifold $S^n$ into $n$-volumes associated with each vertex. A specific tessellation will not be selected here, but a general definition is given below.

\begin{definition}[Vertex volume]
\label{def:V_v}
In a piecewise flat manifold $S^n$, an $n$-dimensional open region $V_v \subset S^n$ dual to a vertex $v \in S^n$ is defined to have the following properties:
\begin{enumerate}
\item $v \in V_v$ and no other vertex is contained within $V_v$.

\item The regions $V_v$ form a complete tessellation of $S^n$,
\begin{equation}
|S^n| = \sum_{v \in S^n} |V_v|, \qquad
 V_{v_i} \cap V_{v_j} = \emptyset \quad
 \forall \ i \neq j .
\end{equation}
with $|S^n|$ and $|V_v|$ representing the $n$-volumes of $S^n$ and $V_v$ respectively.

\item Only hinges $h_v$ in the star of $v$ intersect $V_v$.
\end{enumerate}
\end{definition}

The definition above gives the properties that are required for the consistency of the constructions that follow. While point $3$ is not strictly necessary, with its removal needing only minor corrections to the rest of the paper, it seems like a reasonable requirement for the intuition being used.
%%  Seems like reasonable requirement?...
The total volume of $S^n$ should also be distributed in some consistent manner over all of the regions $V_v$. This is a more difficult property to define however, and there are many interpretations of what `consistency' should mean. The Voronoi and barycentric tessellations, discussed in appendix \ref{sec:Duals}, are each based on different non-complimentary concepts of this consistency for example.

The average piecewise flat scalar curvature over each such volume can now be given.

\begin{theorem}[Piecewise flat scalar curvature]
\label{thm:R}
The average scalar curvature over a dual $n$-volume $V_v$ is
\begin{equation}
\label{Rv}
R_v
 := \widetilde R_{V_v}
 = \frac{2}{|V_v|}
   \sum_{h \subset \mathrm{star}(v)} |h_{|V_v}| \epsilon_h ,
\end{equation}
with $h_{|V_v} = h \cap V_v$, and $|V_v|$ representing the $n$-volume of $V_v$.
\end{theorem}

\begin{proof}
The region $V_v$ can be decomposed into subregions $D_h$, each enclosing the restriction of a hinge $h$ to $V_v$, and intersecting no other hinge,
\begin{equation}
h_{|V_v} \subset D_h, \qquad
 h_i \cap D_{h_j} = \emptyset \quad
 \forall \ i \neq j, \qquad
 |V_v| = \sum_{h \subset \mathrm{star}(v)} |D_h|.
\end{equation}
The integral of the sectional curvature over all of $V_v$ is then given by the sum of the integrals over each $D_h$, which are given by corollary \ref{cor:IntRhD},
\begin{equation}
\int_{V_v} R \ \mathrm{d} V^n
 = \sum_{h \subset \mathrm{star}(v)}
   \int_{D_h} R \ \mathrm{d} V^n
 = 2 \sum_{h \subset \mathrm{star}(v)} |h_{|V_v}| \epsilon_h ,
\end{equation}
assuming small deficit angles. The average curvature is then found by dividing the integral by the volume of $V_v$.

\end{proof}

The total integral of this scalar curvature $R_v$ over a piecewise flat manifold $S^n$ is still equivalent to the Regge action, as shown below. In fact, equation (\ref{Rv}) was derived directly from the Regge action for a Voronoi tessellation in \cite{MMR}.

\begin{corollary}[Total scalar curvature]
\label{cor:IntR}
The total integral of the scalar curvature $R_v$ over a piecewise flat manifold $S^n$ is
\begin{equation}
\int_{S^n} R_v \ \mathrm{d} V^n = 2 \sum_{h \subset S^n} |h| \epsilon_h ,
\label{ReggeEq}
\end{equation}
which is equivalent to the Regge action for $S^n$.
\end{corollary}

\begin{proof}
The integral of $R_v$ over $S^n$ is equal to the sum of the integrals for each vertex volume $V_v$,
\begin{equation}
\int_{S^n} R_v \ \mathrm{d} V^n
 = \sum_{v \in S^n} \left(
 2 \sum_{h \subset \mathrm{star}(v)} |h_{|V_v}| \epsilon_h
 \right) .
\end{equation}
Each part of each hinge $h$ must be contained in a single volume $V_v$, since the volumes $V_v$ form a tessellation of $S^n$ (except for any boundaries which will be of measure zero). Since the deficit angles $\epsilon_h$ are fixed over each hinge $h$, the result follows.
\end{proof}

Constructing the scalar curvature at vertices is also consistent with earlier combinatorial analogues of the scalar curvature. The three-dimensional combinatorial scalar curvature of Cooper \& Rivin \cite{CoopRiv} is defined as the solid angle deficit at a vertex $S_v$, which is shown to be given in terms of deficit angles by
\begin{equation}
S_v = \sum_{h \subset \mathrm{star}(v)} \epsilon_h .
\end{equation}
The triangulations in \cite{CoopRiv} are specifically related to sphere-packings, with radial lengths $r_v$ assigned to each vertex, and the length of each edge defined as the sum of the vertex radii at either end. In this case, the expression $S_v \, r_v$ is equivalent to the integral of $R_v$ over $V_v$ above. This was generalized to a larger class of triangulations by Glickenstein \cite{Glick}, where the edge-lengths are still separated into parts associated with each vertex. The scalar curvature there is defined as $\sum_{h \subset \mathrm{star}(v)} |h_{v}| \epsilon_h$ which is also consistent with the integral of $R_v$ over $V_v$. In fact, these integrated scalar curvatures have been shown to converge to the smooth scalar curvature measure $R \, \mathrm{d} V^n$ by Cheeger, M\"{u}ller \& Schrader \cite{CMS}, for an appropriate convergence of triangulations to the smooth manifold.

\section{Sectional Curvature}
\label{sec:Sec}

The two dimensional case has already been discussed in relation to the scalar curvature. A decomposition of a piecewise flat manifold $S^n$ into areas $A_v$, each enclosing a single vertex $v$, give appropriate regions for the averaging the sectional curvature. The integrated curvature over each $A_v$ is then given by $\epsilon_v$ for small deficit angles, and the average sectional curvature over $A_v$ by
\begin{equation}
^{(2)}K_v := \frac{\epsilon_v}{A_v} .
\label{K2D}
\end{equation}
As with the scalar curvature, for an appropriate distribution of the total area of $S^2$ over each $A_v$, $^{(2)}K_v$ can be compared with the smooth sectional curvature at the point $v_M \in M^n$ corresponding to $v$.

\subsection{Multiple hinges}

In higher dimensions there is a natural choice of orientation with which to compute sectional curvatures at each hinge $h$, given by the surfaces $P_h^\perp$ orthogonal to $h$. For some $n$-volume $D_h$ enclosing all of the hinge $h$ and no other, the average sectional curvature orthogonal to $h$ will be
\begin{equation}
\widetilde K_{D_h} = \frac{|h| \epsilon_h}{|D_h|} ,
\label{singlehK}
\end{equation}
using the same procedure as the proof of lemma \ref{lem:IntRh}. This is equivalent to dividing $\epsilon_h$ by the average cross-sectional area of $D_h$, orthogonal to $h$. This has been suggested as a piecewise flat analogue of the sectional curvature orthogonal to an edge a number of times \cite{HambWill,SRT}.

% In \cite{Miller97} the $2$-surface of the circumcentric lattice dual to $h$ is used to give the area. This is equivalent to the maximum cross-sectional area for the circumcentric dual volumes associated with the edge, or three times the average cross-sectional area.

However, as discussed at the end of section \ref{sec:DefAng}, there is no direct correlation between the deficit angle at a hinge $h$ and the projection of the Riemann curvature orthogonal to $h$. The tessellation of the manifold is also inappropriate, assigning separate volumes to each orientation, where sectional curvatures for each of $\frac{n}{2} (n-1)$ linearly independent planes all occupy the same space in a smooth manifold. Additionally, the results in section \ref{sec:Comp} show a lack of convergence of this expression to the smooth sectional curvature values. This is particularly apparent in the diagonal edges in the $3$-cylinder model, which have zero deficit angle even though the smooth curvature cannot be zero, and in figures \ref{fig:G} and \ref{fig:N3} for the Gowdy and and Nil-$3$ manifolds.

From all of this, it is clear that the sectional curvature requires construction over $n$-volumes containing more than one hinge, and that individual hinges should be considered \emph{super}-local when comparing curvatures with their smooth counterparts. It seems natural to use the vertex volumes $V_v$, similar to the scalar curvature, however the sectional curvature requires an orientation and there is no consistent way of defining almost-flat $2$-surfaces through a vertex, or a foliation of surfaces through $V_v$. Surfaces $P_h^\perp$ orthogonal to a given hinge \emph{can} be defined unambiguously, and in three dimensions these can even be related to an average orientation orthogonal to the geodesic line-segments $\ell_M$ in a smooth manifold.

The surfaces $P_h^\perp$ must first be extended to enclose neighbouring hinges, with this extension defined below.

\begin{definition}[Extension of $2$-surfaces]
\label{def:Pih}
At each point $p \in h$ the surface $P_h (p)$ can be constructed by developing the surface $P^\perp_h$ around both sides of neighbouring hinges $h_i$ which make a well defined angle $\theta_i$ with $h$, forming a union of $P_h^\perp$ with the surfaces $P_{h_i}^{\theta_i}$.
\end{definition}

Intrinsically, this surface is locally isomorphic to $\mathbb{R}^2$ everywhere but at the intersection points with the hinges $h$ and $h_i$. As an embedding within $S^n$, the surface $P_h$ forms corners along a finite number of lines $k_i$ emanating from each hinge $h_i$ in an outward radial direction from $h$. These corners form angles which deviate from $\pi$ by less than $\epsilon_{h_i}$ from lemma \ref{lem:Pitheta}.

\begin{figure}[h]
\begin{center}
\includegraphics[scale=0.3]{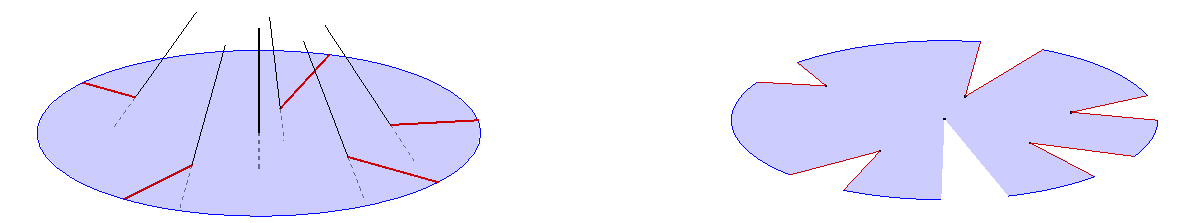}
\end{center}
\vspace{-0.5cm}
\caption{The development of the surface $P^\perp_h$ around neighbouring hinges $h_i$ to produce $P_h$, and the embedding of $P_h$ in $\mathbb{R}^2$ showing the effective deficit angles.}
\label{fig:Pih}
\end{figure}

\begin{lemma}[Extended integrated sectional curvature]
\label{lem:IKpi}
The integrated curvature for some $D \subset P_h$, enclosing hinges $h, h_1, ... , h_k$ is given by the expression
\begin{equation}
{\cal K}_D
 := \int_D K (P_h) \, \mathrm{d} A
 = \epsilon_h + \sum_{i = 1}^k \cos \theta_i \, \epsilon_i + O(\epsilon^2),
 \label{IKPlane}
\end{equation}
with $\theta_i$ giving the angle between each hinge $h_i$ and $h$, and $\epsilon_i$ the deficit angle of $h_i$.
\end{lemma}

\begin{proof}
The region $D$ can be decomposed into regions $D_h \subset P_h^\perp$, $D_i \subset P_{h_i}^{\theta_i}$ and possibly other regions enclosing no hinges at all. The integral of $K (P_h)$ over all of $D$ is equivalent to the sum of the integrals over each of the regions in the decomposition. Assuming small deficit angles, the integrated curvature for these regions is given by $\epsilon_h$, $\cos \theta_i \, \epsilon_i$ and $0$ respectively.

\end{proof}

\subsection{Sectional curvature orthogonal to an edge in $S^3$}

The volumes for constructing the sectional curvature should be formed by the dual tessellation of vertex volumes, however the edges give the best way of defining orientation. To satisfy both criteria, the volume $V_v$ of each vertex in the closure of an edge $\ell$ can be restricted to those points which are orthogonally separated from $\ell$. These restricted volumes for both vertices in $\bar \ell$ can then be combined to give a volume associated with $\ell$, which can be foliated into surfaces $P_\ell$ orthogonal to $\ell$ and is based on the dual tessellation.

\begin{definition}[Edge volume]
\label{def:Vh}

\

\begin{enumerate}
\item For each vertex $v$ in the closure of an edge $\ell$, the restriction of $V_v$ to $\ell$ is defined as the sub-region $V_{v|\ell}$ containing $\ell_{|V_v}$ and bounded by the $2$-surface $P_\ell (v)$ orthogonal to $\ell$ at $v$. This restriction can also be defined by the intersection of $V_v$ with the integral of the $2$-surfaces $P_\ell$ over $\ell$,
\begin{equation}
V_{v|\ell}
 := V_v \, \cap \int_\ell P_\ell \, \mathrm{d} p .
\end{equation}

\item The volume $V_\ell$ associate with each edge $\ell$ is defined as the sum of the restricted vertex volumes $V_{v|\ell}$ for each $v$ in the closure of $\ell$, or equivalently by the intersection
\begin{equation}
V_\ell
 := \left(\cup_v \, V_v\right) \cap \int_\ell P_\ell \, \mathrm{d} p ,
 \quad s.t. \quad
 v \in \bar \ell .
\end{equation}
\end{enumerate}
\end{definition}

In two dimension, where the hinges are given by vertices, the definition of the volume above is consistent with the areas associated with each vertex $A_v$. The volume $V_\ell$ will be used below to give the piecewise flat sectional curvature orthogonal to each edge $\ell \subset S^3$.

\begin{theorem}[Piecewise flat sectional curvature]
\label{thm:Kl}
In a three-dimensional piecewise flat manifold $S^3$, the average sectional curvature orthogonal to an edge $\ell$ over the volume $V_\ell$ is
\begin{equation}
K_\ell := \widetilde K_{V_\ell}
 = \frac{1}{|V_\ell|} \left(
     |\ell| \epsilon_\ell
   + \sum_{h} |h_{|V_\ell}| \, \cos^2 \theta_h \, \epsilon_h
   \right) ,
\end{equation}
with the sum taken over the other edges (hinges) $h$ intersecting $V_\ell$, with $h_{|V_\ell} = h \cap V_\ell$, $\theta_h$ representing the angle between $\ell$ and $h$, and $\epsilon_h$ the deficit angle at $h$.
\end{theorem}

\begin{figure}[h]
\begin{center}
\includegraphics[scale=0.25]{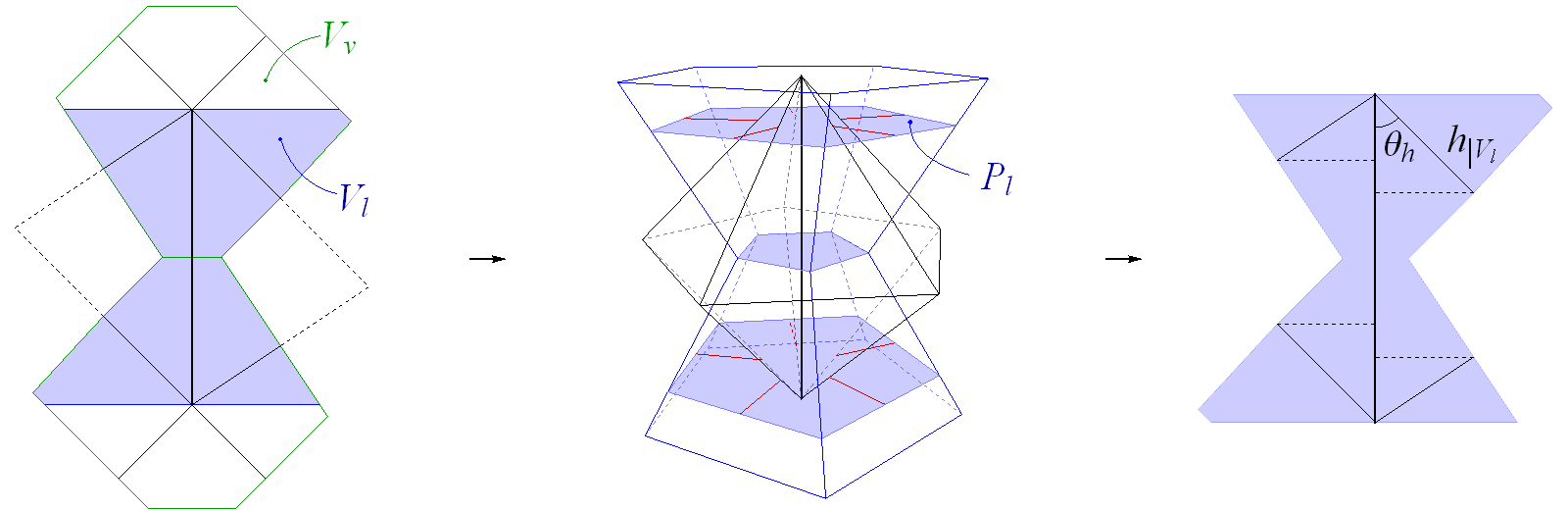}
\end{center}
\vspace{-0.5cm}
\caption{A cross-section of the region around an edge $\ell$ showing the volume $V_\ell$ on the left, followed by the outline of $V_\ell$ and a number of surfaces $P_\ell$ in the center, and a cross-section of $V_\ell$ again with the right angled triangles for each $h_{|\ell}$ shown on the right.}
\label{fig:Kl}
\end{figure}

\begin{proof}
At each point $p \in \ell$ the integral of the sectional curvature over the $2$-surface region $D(p) = P_\ell(p) \cap V_\ell$ is given, from lemma \ref{lem:IKpi}, by the expression
\begin{equation}
{\cal K}_D (p)
 = \epsilon_\ell + \sum_h \cos \theta_h \, \epsilon_h ,
\end{equation}
for each hinge $h$ intersecting $D(p)$. The integral of the sectional curvature orthogonal to $\ell$ over the entire volume $V_\ell$ is then found by the integrating ${\cal K}_D (p)$ over the length of $\ell$,
\begin{equation}
\int_{V_\ell} K(P_\ell) \, \mathrm{d} V
 = \int_{\ell} {\cal K}_D (p) \, \mathrm{d} p .
\end{equation}
Since each ${\cal K}_D (p)$ is a sum of contributions from different hinges $h$, the full volume integral can be given by a sum of integrals for each of these hinges
\begin{equation}
\int_{V_\ell} K(P_\ell) \, \mathrm{d} V
 = \epsilon_\ell \, |\ell|
 + \sum_h \int_{\ell} \cos \theta_h \, \epsilon_h \, \mathrm{d} p
 = \epsilon_\ell \, |\ell|
 + \sum_h |\ell_{|h}| \, \cos \theta_h \, \epsilon_h ,
\end{equation}
with the last term due to $\theta_h$ and $\epsilon_h$ not depending on the points $p \in \ell$. Here $|\ell_{|h}|$ represents the length of $\ell$ for which the hinge $h_h$ intersects the $2$-surface regions $D(p)$. This gives a right angled triangle with $h_{|V_\ell}$ as the hypothenuse, as shown in the right-hand diagram in figure \ref{fig:Kl}. The length $|\ell_{|h}|$ is therefore equal to $|h_{|V_\ell}| \cos \theta_h$, and so
\begin{equation}
\int_{V_\ell} K(P_\ell) \, \mathrm{d} V
 = \epsilon_\ell \, |\ell|
 + \sum_h |h_{|V_\ell}| \, \cos^2 \theta_h \, \epsilon_h .
\end{equation}
Dividing by the volume of $V_\ell$ gives the average sectional curvature orthogonal $\ell$ over $V_\ell$.

\end{proof}

The expression for $K_\ell$ can also be simplified a little by re-defining the label $h$ to include the edge $\ell$,
\begin{equation}
K_\ell 
 = \frac{1}{|V_\ell|} \sum_{h|h \cap V_\ell} |h_{|V_\ell}| \, \cos^2 \theta_h \, \epsilon_h ,
\end{equation}
with $\ell_{|V_\ell} = \ell$ and $\theta_\ell := 0$. An average cross-sectional area of $V_\ell$ orthogonal to $\ell$ can also be defined as
 $A_\ell := |V_\ell| / |\ell|$,
 giving a further interpretation of $K_\ell$ as a weighted sum of deficit angles divided by this area
\begin{equation}
K_\ell
 = \frac{\epsilon_\ell}{A_\ell}
 + \sum_h \frac{|h_{|V_\ell}| \cos^2 \theta}{|\ell|}
   \frac{\epsilon_h}{A_\ell} .
\end{equation}
This equation is of a similar form to the single-edge curvature (\ref{singlehK}), with a newly-defined area and higher order contributions from the nearby hinges $h$.

The construction in theorem \ref{thm:Kl} compares much more favourably with the smooth sectional curvatures of the manifolds tested in section \ref{sec:Comp}. It gives exact values for all of the edges in the $3$-cylinder, and the representative graphs in figures \ref{fig:G} and \ref{fig:N3} show a clear convergence to the smooth Gowdy and Nil-$3$ sectional curvatures, as do the percentage errors in tables \ref{tab:GErr} and \ref{tab:N3Err}.

%%  Suggestion for higher dimensions?...

%This expression should be easily adapted to higher dimensions for the sectional curvature orthogonal to a hinge, though it is still unclear whether the relation $|\ell_{|h}| = |h_{|V_\ell}| \cos \theta_h$ has a higher dimensional analogue. There is also an ambiguity in the choice of surface in the smooth manifold with which to relate the piecewise flat values.

\section{Ricci Curvature and Ricci Flow}
\label{sec:Ric}

\subsection{Ricci curvature along edges}

At a point $p$ in an $n$-dimensional Riemannian manifold $M^n$, the Ricci curvature along some unit vector $\vec{u} \in T_p M$ can be given by the sum of the sectional curvatures for a set of $n - 1$ orthogonal 2-planes, all of which contain the vector $\vec{u}$. For an orthonormal basis $\{ \vec{e_1}, \vec{e_2}, ... , \vec{e_n} \}$ of $T_p M$,
\begin{equation}
\tens{Rc}(\vec{e_1}, \vec{e_1})
 = \sum_{i = 1}^n \left< \tens{Rm}(\vec{e_i}, \vec{e_1}) \vec{e_1},
   \vec{e_i} \right>
 = \sum_{i = 2}^n K(\vec{e_i}, \vec{e_1}) ,
\label{RcK}
\end{equation}
with the final part following directly from the definition of the sectional curvature (\ref{Sec}), with $K(\vec{e_1}, \vec{e_1})$ clearly vanishing. Off-diagonal terms of the Ricci tensor can also be found in terms of the Ricci tensor along a certain combination of vectors. For two vectors $\vec{u}, \vec{v} \in T_p M^n$, an infinitesimal triangle can be formed using the vector $\vec{u} - \vec{v} \in T_p M^n$, and by the bilinearity and symmetry of the Ricci tensor,
\begin{eqnarray}
&&\tens{Rc}(\vec{u} - \vec{v}, \vec{u} - \vec{v})
 = \tens{Rc}(\vec{u}, \vec{u})
 - 2 \tens{Rc}(\vec{u}, \vec{v})
   + \tens{Rc}(\vec{v}, \vec{v})
 \nonumber \\ \Leftrightarrow \qquad
&&\tens{Rc}(\vec{u}, \vec{v})
 = \frac{1}{2}\left(
     \tens{Rc}(\vec{u}, \vec{u})
   + \tens{Rc}(\vec{v}, \vec{v})
   - \tens{Rc}(\vec{u} - \vec{v}, \vec{u} - \vec{v})
     \right) .
\label{Rcuv}
\end{eqnarray}
The entire Ricci tensor at $p$ can therefore be found from the Ricci curvature along a set of $n$ linearly independent vectors in $T_p M$, and along the $\frac{n}{2} (n - 1)$ vectors giving the difference between each pair of these.

For each $n$-simplex $\sigma^n$ in a piecewise flat manifold $S^n$, there are exactly $n$ edges at each vertex and another $\frac{n}{2} (n - 1)$ giving the difference between each pair, since $\bar \sigma^n$ is a Euclidean space. This gives the exact configuration required to completely determine a symmetric bilinear tensor within each $\sigma^n$. If an analogue of the Ricci tensor can be found along each edge $\ell \subset S^n$, then a Ricci tensor can be constructed within each $n$-simplex, and coordinate transformations performed to compare with a given coordinate system in a smooth manifold.

However it seems more convenient to simply compare the values along each edge $\ell \subset S^n$ with the average value of the Ricci tensor tangent to the corresponding geodesic segments $\ell_M \subset M^n$. This correspondence is coordinate independent, and avoids any issues associated with the choosing of coordinate charts. Over the geodesic segments $\ell_M \subset M^n$, the average value of $\tens{Rc}$ along the vector field $\vec{\hat \ell}_M$ tangent to $\ell_M$ can be found by integrating along $\ell_M$ and dividing by its length,
\begin{equation}
\widetilde{\tens{Rc}}_{\ell_M}
 := \frac{1}{|\ell_M|} \int_{\ell_M}
 \tens{Rc} (\vec{\hat \ell}_M, \vec{\hat \ell}_M) \
 \sqrt{\tens{g}_M \left(\frac{d \ell_M}{d s}, \frac{d \ell_M}{d s}\right)}
 \ \mathrm{d} s ,
\label{RclM}
\end{equation}
with $s$ giving a parameterization of $\ell_M$.

\subsection{Three dimensional Ricci curvature}

In three dimensions, the Ricci curvature along a vector in a smooth manifold $M^n$ can be given by both the scalar and sectional curvatures.

\begin{lemma}[Smooth Ricci curvature in three dimensions]
\label{lem:RcM3}
The Ricci curvature along a unit vector $\vec{u} \in T_p M^3$ can be given in terms of the scalar curvature $R$ and the sectional curvature $K^\perp (\vec{u})$ of the plane orthogonal to $\vec{u}$ by the expression
\begin{equation}
^{(3)}\tens{Rc} (\vec{u}, \vec{u})
 = \frac{1}{2} R - K^\perp (\vec{u}) .
\label{RcM3}
\end{equation}
\end{lemma}

\begin{proof}
The scalar curvature in three dimensions, defined as the trace of the Ricci curvature, can be given in terms of a sum of sectional curvatures using (\ref{Scal}),
\begin{equation}
R = \sum_{i = 1}^n \tens{Rc}(\vec{e_i}, \vec{e_i})
 = 2 \sum_{i < j} K(\vec{e_i}, \vec{e_j})
 = 2 \left( K(\vec{e_1}, \vec{e_2}) + K(\vec{e_1}, \vec{e_3}) +
 K(\vec{e_2}, \vec{e_3}) \right) .
\end{equation}
The Ricci curvature along $\vec{e_1}$ in three dimensions is then given from (\ref{RcK}) as
\begin{equation}
^{(3)}\tens{Rc} (\vec{e_1}, \vec{e_1})
 = K(\vec{e_1}, \vec{e_2}) + K_(\vec{e_1}, \vec{e_3})
 = \frac{1}{2} R - K_(\vec{e_2}, \vec{e_3}) .
\end{equation}
Since the plane $P(\vec{e_2}, \vec{e_3})$ is entirely determined in three dimensions by the vector $\vec{e_1}$ orthogonal to it, the Ricci curvature along any unit vector $\vec{u}$ can be defined in terms of the sectional curvature $K^\perp (\vec{u})$ of the plane orthogonal to $\vec{u}$.

\end{proof}

Since piecewise flat analogues of the scalar curvature, and of the sectional curvature orthogonal to an edge $\ell \subset S^3$, have already been given, the expression in lemma \ref{lem:RcM3} can be used to give a piecewise flat Ricci curvature along $\ell$.

\begin{theorem}[Piecewise flat Ricci curvature]
\label{thm:Rcl}
In a piecewise flat manifold $S^3$, the average Ricci curvature along an edge $\ell$ is
\begin{equation}
\tens{Rc}_\ell := \widetilde{\tens{Rc}} (\ell) = \frac{1}{2 |\ell|} \left(|\ell_{|V_1}| R_{v_1} + |\ell_{|V_2}| R_{v_2}\right) - K_\ell ,
\end{equation}
with $\ell_{|V_1} = \ell \cap V_{v_1}$ and $\ell_{|V_2} = \ell \cap V_{v_2}$ for the vertices $v_1, v_2 \in \bar \ell$.
\end{theorem}

\begin{proof}
The result follows from lemma \ref{lem:RcM3}, with the value of the scalar curvature at each point $p \in \ell$ given by the value $R_v$ for the volume $V_v \ni p$.

\end{proof}

While this construction does not use the same approach as both the scalar and sectional curvatures, such a construction is outlined in appendix \ref{sec:RcSn}, which also extends to arbitrary dimensions. Unfortunately, an appropriate volume over which to compute the Ricci curvature has not yet been found even in three dimensions, leaving this approach incomplete at present.

\begin{remark}
Interestingly, equation (\ref{RcM3}) can be seen as a component version of the expression for the Einstein tensor in three dimensions. While the variation of the Regge action gives a piecewise flat analogue of the Einstein field equations as $\epsilon_h = 0$, this equation now gives a piecewise flat construction that converges to the smooth Einstein tensor. Although for the vacuum equations of general relativity $\epsilon = 0 \Leftrightarrow -K_{\ell} = 0$, any inclusion of matter will likely require an Einstein tensor that converges to the smooth value.
\end{remark}

\subsection{Ricci flow}
\label{sec:RF}

On a smooth time-varying manifold $M^n$, the Ricci flow acts to uniformize the metric \cite{ChowKnopfRF}, giving a rate of change of the metric in terms of the Ricci curvature tensor
\begin{equation}
\frac{d \tens{g}_M}{d t}
 = - 2 \tens{Rc} .
\label{RF}
\end{equation}
A normalized Ricci flow, keeping the total volume constant, is more useful in many situations and was actually the form first introduced by Hamilton \cite{HamRf},
\begin{equation}
\frac{d \tens{g}_M}{d t}
 = - 2 \tens{Rc} + \frac{2}{n} \widetilde{R}_M \, \tens{g}_M ,
\label{RFn}
\end{equation}
with $\widetilde{R}_M$ giving the average of the scalar curvature over the whole manifold. The effect of these flows on the lengths of geodesic segments can be given in a straight forward way.

\begin{lemma}[Smooth geodesic Ricci flow]
\label{lem:RFlM}
The fractional rate of change of the length of a geodesic segment $\ell_M$ under normalized Ricci flow in a smooth manifold $M^n$ is
\begin{equation}
\frac{1}{|\ell_M|} \frac{d |\ell_M|}{d t}
 = - \widetilde{\tens{Rc}}_{\ell_M}
   + \frac{1}{n} \widetilde{R}_M ,
\label{RFlM}
\end{equation}
where $\widetilde{\tens{Rc}}_{\ell_M}$ represents the average of the component of $\tens{Rc}$ tangent to $\ell_M$ over the length of $\ell_M$, as in equation (\ref{RclM}).
\end{lemma}

\begin{proof}
From the equation for the geodesic length (\ref{Geodesic}), the rate of change of a geodesic segment $\ell_M$ is
\begin{equation}
\frac{d |\ell_M|}{d t}
 = \frac{d}{d t} \int_{\ell_M}
 \sqrt{\tens{g}_M
 \left(\frac{d \ell_M}{d s}, \frac{d \ell_M}{d s}\right)}
 \ \mathrm{d} s
% = \int_\ell_M \frac{d}{d t}
% \sqrt{\tens{g}_M
% \left(\frac{d \ell_M}{d s}, \frac{d \ell_M}{d s}\right)}
% \ \mathrm{d} s
 = \int_{\ell_M}
 \frac{
 \frac{d}{d t} \tens{g}_M
       \left(\frac{d \ell_M}{d s}, \frac{d \ell_M}{d s}\right)}
 {2 \sqrt{\tens{g}_M
       \left(\frac{d \ell_M}{d s}, \frac{d \ell_M}{d s}\right)}
 }
 \ \mathrm{d} s ,
\end{equation}
for some parameter $s$ of $\ell_M$, using the chain rule for the last part. For normalized Ricci flow, the derivative of the metric component on the top line is given by (\ref{RFn}), so that
\begin{equation}
\frac{d |\ell_M|}{d t}
 = \int_{\ell_M}
 \frac{
 - \tens{Rc} \left(\frac{d \ell_M}{d s}, \frac{d \ell_M}{d s}\right)
 + \frac{1}{n} \widetilde{R}_M \,
  \tens{g}_M \left(\frac{d \ell_M}{d s}, \frac{d \ell_M}{d s}\right)}
 {\sqrt{\tens{g}_M
       \left(\frac{d \ell_M}{d s}, \frac{d \ell_M}{d s}\right)}
 }
 \ \mathrm{d} s .
\end{equation}
The unit vectors $\vec{\hat \ell_M}$ tangent to $\ell_M$ differ only in scale from the vectors $\frac{d \ell_M}{d s}$. Multiplying above and below inside the integral, by the square of the length of $\frac{d \ell_M}{d s}$,
\begin{eqnarray}
\frac{d |\ell_M|}{d t}
 &=& \int_{\ell_M}
 \left(
 - \tens{Rc} \left(\vec{\hat \ell}_M, \vec{\hat \ell}_M\right)
 + \frac{1}{n} \widetilde{R}_M
 \right)
 \sqrt{\tens{g}_M \left(\frac{d \ell_M}{d s}, \frac{d \ell_M}{d s}\right)}
 \ \mathrm{d} s \\
 &=& - |\ell_M| \ \widetilde{\tens{Rc}}_{\ell_M}
     + \frac{1}{n} \widetilde{R}_M \, |\ell_M| ,
\end{eqnarray}
with $\widetilde{\tens{Rc}}_{\ell_M}$ giving the average value of $\tens{Rc} \left(\vec{\hat \ell}_M, \vec{\hat \ell}_M\right)$ over the length of $\ell_M$, from (\ref{RclM}).

\end{proof}

This shows that the fractional change in the lengths of geodesic segments depends only on the average Ricci curvature tangent to the edges, and the average of the total scalar curvature. This implies the change in edge-lengths of triangulations of a fixed simplicial complex on $M^n(t)$, also depend only on the average Ricci curvature tangent to the corresponding geodesic segments. Since this is precisely what is being approximated by $\tens{Rc}_\ell$, a piecewise flat Ricci flow can be given, noting that the average scalar curvature over the entire piecewise flat manifold $S^n$ is
\begin{equation}
\widetilde R_S := \frac{2}{|S^n|} \sum_{h \subset S^n} |h| \epsilon_h ,
\end{equation}
using corollary \ref{cor:IntR}, with $|S^n|$ representing the full $n$-volume of $S^n$.

\begin{theorem}[Piecewise flat Ricci flow]
\label{prop:RFS}
For a piecewise flat triangulation $S^3$ of a smooth manifold $M^3$, the evolution of the edge lengths $\ell \subset S^3$ by the equation
\begin{equation}
\frac{1}{|\ell|} \frac{d |\ell|}{d t}
 = - \mathrm{Rc}_\ell
   + \frac{1}{n} \widetilde{R}_S ,
\label{RFS}
\end{equation}
gives a piecewise flat approximation of the evolution of $M^3$ under Ricci flow, as long as the simplicial complex of $S^3$ remains a good triangulation for $M^3$.
\end{theorem}

\begin{proof}
For a fixed simplicial complex on $M^3$, the triangulation $S^3 (t)$ of the evolution of $M^3$ is completely determined by the evolution of the geodesic segments $\ell_M \subset M^3$. It follows from lemma \ref{lem:RFlM} that this evolution depends only on the average Ricci curvature along the geodesic segments. As long as the piecewise flat Ricci curvature $\mathrm{Rc}_\ell$ and average scalar curvature $\widetilde{R}_S$ give close approximations for $\widetilde{\mathrm{Rc}}_{\ell_M}$ and $\widetilde{R}_M$,
\begin{equation}
\frac{1}{|\ell|} \frac{d |\ell|}{d t}
 = \frac{1}{|\ell_M|} \frac{d |\ell_M|}{d t}
 = - \widetilde{\tens{Rc}}_{\ell_M}
   + \frac{1}{n} \widetilde{R}_M
 \simeq - \mathrm{Rc}_\ell
        + \frac{1}{n} \widetilde{R}_S ,
\end{equation}
since the lengths of the edges $|\ell| := |\ell_M|$ by definition.
\end{proof}

There is no guarantee that the simplicial complex of $S^3$ will continue to give a good approximation for $M^3 (t)$ for late times $t$, however the size of the deficit angles $\epsilon_h$ of $S^3 (t)$ should give a reasonable measure of this. The non-normalized Ricci flow can easily be given by dropping the second term on the right hand sides of both (\ref{RFlM}) and (\ref{RFS}).

\section{Computations}
\label{sec:Comp}

The piecewise flat curvature and Ricci flow analogues have been tested on triangulations of a number of manifolds. These begin with the $3$-sphere and $3$-cylinder which have high degrees of symmetry, with the curvature inherited from the $2$-sphere in the latter case. These are followed by the more complicated geometries of the Gowdy and Nil-$3$ manifolds, each having both positive and negative sectional curvature and specific challenges for triangulating.

For all of the triangulations used, the edges (hinges) $h$ intersecting each volume $V_\ell$ are all contained in the closure of $\mathrm{star}(\ell)$, i.e. they all share a tetrahedron with $\ell$. This is not unusual for triangulations containing tetrahedra that are somewhat close to equilateral, and it means that all of the edges used in the construction for $K_\ell$ share a triangle with $\ell$. The area of each triangle $\sigma^2 \subset \mathrm{star}(\ell)$ can be found in two ways
\begin{equation}
\mathrm{Area}(\ell, h_1, h_2)
 = \frac{1}{2} \, |\ell| \, z
 = \frac{1}{4} \sqrt{
   (\ell + h_1 + h_2) (\ell + h_1 - h_2) (\ell + h_2 - h_1) (h_1 + h_2 - \ell)} ,
\label{CompKl}
\end{equation}
where $z$ is the perpendicular distance from $\ell$ to the third vertex of $\sigma^2$, see figure \ref{fig:Ktri}. The length of $z$ can be found in terms of edge-lengths using the second equation for the triangle area above, with Pythagoras theorem then giving the length $d$, with $\cos \theta = d/|h_1|$.

\begin{figure}[h]
\begin{center}
\includegraphics[scale=0.2]{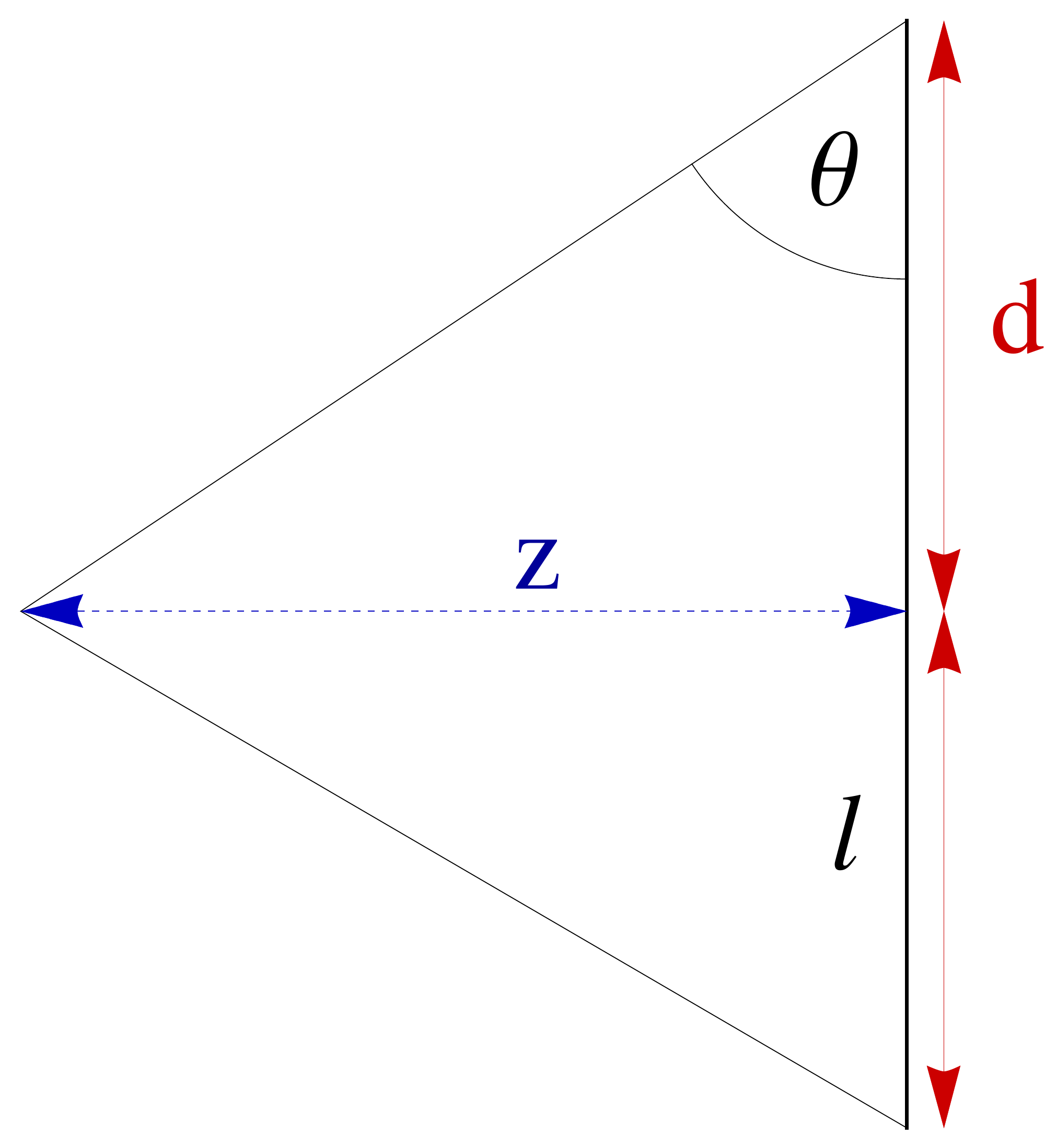}
\end{center}
\vspace{-0.5cm}
\caption{A triangle containing the edge $\ell$, with the lengths $z$ and $d$ outlined.}
\label{fig:Ktri}
\end{figure}

The dual volumes $V_v$ are also chosen to be either Voronoi or barycentric, so that the lengths $|h_{|V_\ell}| = \frac{1}{2}|h|$. Due to all of this, the equation for the sectional curvature orthogonal to an edge $\ell$ from theorem \ref{thm:Kl} can be given in a simplified form
\begin{equation}
K_\ell
 = \frac{1}{|V_\ell|} \left(
     |\ell| \epsilon_\ell
   + \frac{1}{2} \sum_{\sigma^2 \subset \mathrm{star}(\ell)}
     \left(
       \frac{d^2}{|h_1|} \, \epsilon_1
     + \frac{(|\ell| - d)^2}{|h_2|} \, \epsilon_2
     \right)
   \right) ,
\label{KlRed}
\end{equation}
with $h_1$ and $h_2$ making up the remaining sides of each triangle $\sigma^2$, and their deficit angles given by $\epsilon_1$ and $\epsilon_2$ respectively. It is also important that the angles $\theta_h$ do not increase far beyond $\pi/2$. Though conditional statements can be included for such cases, the differences have been negligible for the few situations that have arisen here.

\subsection{3-Sphere}
\label{sec:3Sph}

A $3$-sphere $M^3_{Sph}$ is a homogeneous, isotropic and compact manifold which is completely determined by its radius $r$. The sectional curvature is invariant to the orientation of a $2$-plane or location on $M^3_{Sph}$ and given by $K_{Sph} = \frac{1}{r^2}$. This then gives the Ricci curvature along any vector as $\tens{Rc}_{Sph} = \frac{2}{r^2}$, and the scalar curvature as $R_{Sph} = \frac{6}{r^2}$. The Ricci flow equations can also be reduced to an evolution equation for the radius, with $\frac{d r}{d t} = - \frac{2}{r}$.

Due to its high level of symmetry, a $3$-sphere can be triangulated in a regular way using equilateral tetrahedra.
% Symmetry transformations can then interchange any pair of $k$-simplices in any orientation.
 For such a triangulation the curvature and Ricci flow need only be computed at a single vertex or edge, with all of the others giving equivalent results. The dihedral angle at an edge of an equilateral tetrahedron in Euclidean space is always $\theta_s = \arccos \left(\frac{1}{3}\right) \simeq 70.53^{\circ}$. As a result, symmetric triangulations can be formed with anywhere from two to five equilateral tetrahedra in the closure of each edge, with any more giving negative deficit angles. The first of these contains only two tetrahedra but is in some sense degenerate, see for example \cite{DoubleTet}. The remaining triangulations are given by $5$, $16$ and $600$-cell models, with the particulars of each given in table \ref{tab:cells}.

\begin{table}[h]
\centering
\begin{tabular}{|r||c|c|c||c|c|c|}
\hline
Model &
 Triangles & Edges & Vertices &
 $\sigma^3 \subset \mathrm{star}(\ell)$ &
 $\sigma^3 \subset \mathrm{star}(v)$ &
 $\ell \subset \mathrm{star}(v)$ \\
\hline
$5$-cell   &   $10$ &  $10$ &   $5$  &  $3$ &  $4$ &  $4$ \\
$16$-cell  &   $32$ &  $24$ &   $8$  &  $4$ &  $8$ &  $6$ \\
$600$-cell & $1200$ & $720$ & $120$  &  $5$ & $20$ & $12$ \\
\hline
\end{tabular}
\caption{The number of each type of simplex in each model, the number of tetrahedra around each edge and vertex, and the number edges at each vertex.}
\label{tab:cells}
\end{table}

The edge-lengths for each triangulation can be given in terms of the radius $r$ of the smooth $3$-sphere, with these shown in table \ref{tab:Sph}. The choice of dual tessellation is also unambiguouss since both the circumcenter and barycenter coincide with the symmetric center of an equilateral tetrahedron. The only term that cannot be computed easily is the restriction of a vertex volume $V_v$ to an edge $\ell$. To approximate this, the average radius is first found by dividing $V_v$ by the solid angle at a vertex, with this radius then multiplied by the solid angle of a hemisphere in Euclidean space. This gives a valid approximation since the volume $V_v$ forms a regular polygon, with the restricted volume $V_{v|\ell}$ approximating a smooth hemisphere. For an equilateral tetrahedron, the solid angle at a vertex is $\alpha_v = \arccos \left( \frac{23}{27} \right)$ and the volume $V_{\sigma^3} = \frac{|\ell|^3}{6 \sqrt{2}}$. This gives
\begin{equation}
V_{v|\ell}
 \simeq 2 \pi \frac{V_v}{s_v \alpha_v}
  = \frac{2 \pi}{s_v \alpha_v} \frac{s_v}{4} V_{\sigma^3}
  = \frac{\pi \, |\ell|^3}{12 \, \sqrt{2} \, \alpha_v} ,
\label{3sphVv}
\end{equation}
with $s_v$ giving the number of tetrahedra in the star of the vertex $v$. With the edge lengths and deficit angles equivalent for all edges throughout each triangulation, the sectional curvature equation (\ref{KlRed}) then reduces to
\begin{equation}
K_\ell
 = \frac{1}{2 \, V_{v|\ell}}
     \left(1 + \frac{s_\ell}{4} \right) |\ell| \, \epsilon
 = \frac{6 \, \sqrt{2} \, \alpha_v}{\pi \, |\ell|^2}
     \left(1 + \frac{s_\ell}{4} \right) \epsilon ,
\end{equation}
with $s_\ell$ representing the number of triangles in the star of $\ell$. This is equivalent to the number of tetrahedra in the star of $\ell$, which can be found in table \ref{tab:cells} for each of the three triangulations.

The Ricci curvature and Ricci flow equation can be found directly from the scalar and sectional curvatures. The results for the curvatures for each triangulation are displayed in table \ref{tab:Sph}, and can be seen to converge to the smooth manifold values as the triangulations become more dense and the deficit angles decrease. The Ricci flow equations can be reduced to a differential equation for the radius $r$ in both the smooth and piecewise flat manifolds, with the piecewise flat equations also converging to the smooth expression for triangulations with increasing resolution.

\begin{table}[h]
\centering
\begin{tabular}{|r||c|c||c|c|c||c|}
\hline
Model & $|\ell|$ & $\epsilon$ &
 $R_v$ & $K_\ell$ & $\tens{Rc}_\ell$ &
 $d r / d t$ \\
\hline
$5$-cell   & $3.22 r$ & $2.59$ &
 $8.46045 /r^2$ & $0.649529 /r^2$ &
 $3.5807 /r^2$  & $-3.5807 /r$ \\
$16$-cell  & $2.18 r$ & $1.36$ &
 $7.23104 /r^2$ & $0.845933 /r^2$ &
 $2.76959 /r^2$ & $-2.76959 /r$ \\
$600$-cell & $0.65 r$ & $0.128$ &
 $6.12124 /r^2$ & $1.00702 /r^2$ &
 $2.0536 /r^2$ & $-2.0536 /r$ \\
\hline
$3$-sphere &  &   &
 $6 /r^2$ & $1 /r^2$ &
 $2 /r^2$ & $-2 /r$ \\
\hline
\end{tabular}
\caption{Edge-lengths and deficit angles for the $3$-sphere triangulations, with piecewise flat curvatures and the Ricci flow equation in terms of the smooth radius $r$.}
\label{tab:Sph}
\end{table}

\begin{remark}
It should be noted that the approximations for the Ricci curvature values are slighly closer for a sectional curvature based on a single edge (\ref{singlehK}), with the deficit angle divided by three times the average hybrid volume associated with each edge. This is equivalent to the sectional curvature definition in \cite{SRF}, giving the same results as the `simplicial Ricci flow' there. The factor of three for the hybrid volume removes the issues with tessellating a space for the sectional curvature, but likely only gives correct results due to the high level of symmetry of the triangulations. The slight improvement over the current approach likely comes from the approximations of the volume $V_{v|\ell}$ in equation (\ref{3sphVv}). Deeper problems with a single-hinge sectional curvature definition become more apparent in the remaining manifolds tested however.
\end{remark}

\subsection{3-Cylinder}
\label{sec:3Cyl}

A $3$-cylinder is given by the direct product of a $2$-sphere with $\mathbb{R}^1$, giving a line element of
\begin{equation}
\mathrm{d} s^2
 = r^2 \mathrm{d} \theta^2
 + r^2 \sin^2 \theta \, \mathrm{d} \phi^2
 + \mathrm{d} z^2 .
\label{3Cyl}
\end{equation}
The intrinsic curvature comes directly from the $2$-dimensional curvature of the $2$-sphere, with $K_{Cyl}^\theta = \frac{1}{r^2} \cos^2 \theta$ for a $2$-plane making an angle $\theta$ with the tangent space of the spherical part. This then gives a scalar curvature of $R_{Cyl} = \frac{2}{r^2}$ and the Ricci curvature along a vector $\vec{u}$ as $\tens{Rc}_{Cyl}^\theta = \frac{1}{r^2} \sin^2 \theta$ where $\vec{u}$ makes an angle $\theta$ with the $z$-coordinate vector.

The most simple triangulation of the $3$-cylinder uses a regular icosahedron for each $2$-sphere, formed by $20$ equilateral triangles. The vertices of neighbouring icosahedra are then joined by line segments orthogonal to each of the triangles in the star of each vertex. This gives a decomposition into a set of Euclidean prisms, which can each be subdivided into three equal-volume tetrahedra by including diagonals of the rectangular sides as edges.

\begin{figure}[h]
\begin{center}
\includegraphics[scale=0.25]{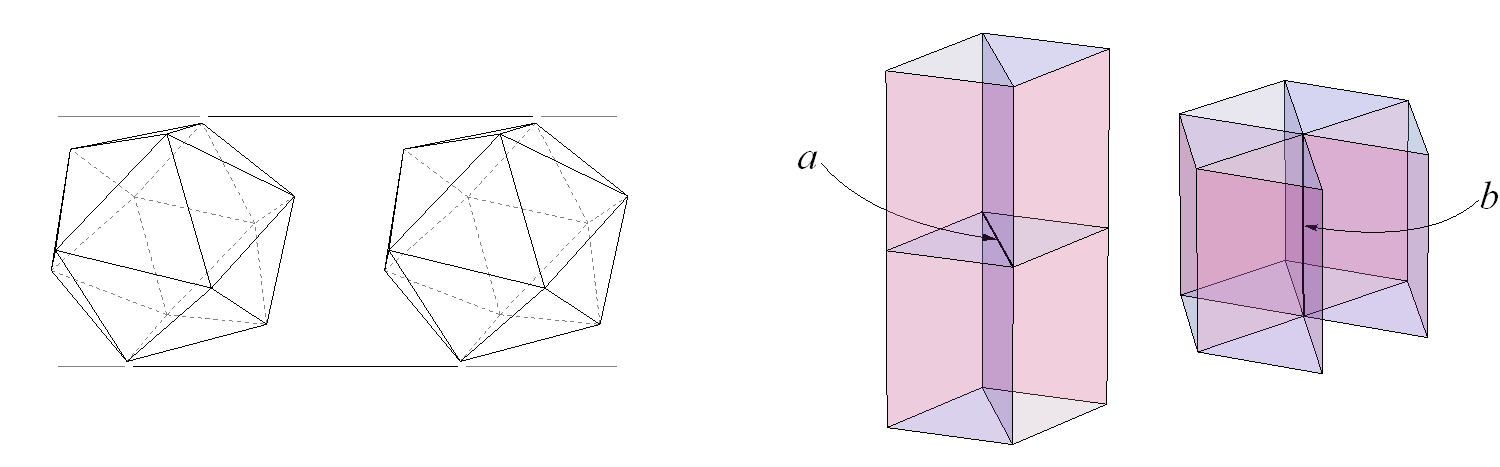}
\end{center}
\vspace{-0.5cm}
\caption{Two icosahedra and the prisms in the star of both an icosahedral edge $a$ and orthogonal edge $b$.}
\label{fig:3CylPrism}
\end{figure}

The decomposition into prisms alone can be used as a piecewise flat decomposition, if the orthogonality between the icosahedral and non-icosahedral edges, denoted $a$ and $b$ respectively, is enforced. The deficit angles $\epsilon_a$ are zero from figure \ref{fig:3CylPrism}, while the edges $b$ have the same deficit angle as the vertex on the $2$-sphere triangulation, $\epsilon_b = \pi/3$. Both the maximum and average cross sectional area of a prism-based Voronoi volume dual to $b$ are also equivalent to the Voronoi area of a vertex in $S^2_{Sph}$, giving the sectional curvature $K_b$ equal to that of a vertex in an icosahedron. These then match the sectional curvatures orthogonal to $a_M$ and $b_M$ in the smooth manifold. The prism-based piecewise flat decomposition was used effectively in \cite{SRF} to give the Ricci flow of a $3$-cylinder, and was also extended in \cite{MNeckPinch1,MNeckPinch2} to non-fixed radial lengths along $\mathbb{R}^1$, producing a dumbbell model, where a good approximation of the Ricci flow behaviour was also found.

\begin{figure}[h]
\begin{center}
\includegraphics[scale=0.25]{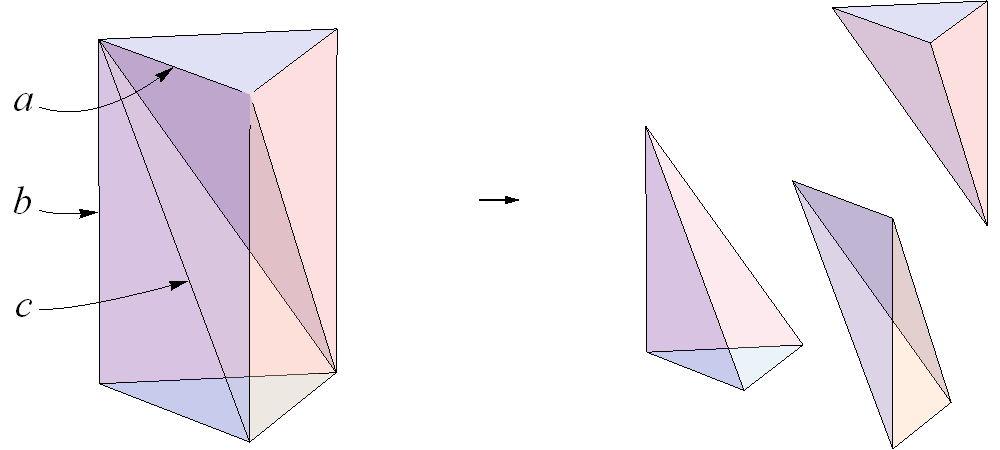}
\end{center}
\vspace{-0.5cm}
\caption{Separation of a prism into three equal volume tetrahedra.}
\label{fig:3CylTet}
\end{figure}

In order to remove the enforcement of the orthogonality condition however, the prisms must be subdivided into tetrahedra. Each prism can easily be divided into three equal volume tetrahedra by simply including diagonal edges, denoted $c$, along the rectangular sides in a consistent manner. These edges have deficit angles of zero, even though the corresponding edges in the smooth manifold have sectional curvatures of $K_{Cyl}^\theta = \frac{1}{r^2} \cos^2 \theta$. Using equation (\ref{CompKl}) however, the sectional curvature orthogonal to each edge $c$ is
\begin{equation}
K_c = \frac{|b| \, \cos^2 \theta \, \epsilon_b}{V_c} \, ,
\end{equation}
with $\theta$ giving the angle between the edges $c$ and $b$.

The diagonal edges $c$ must also be consistent across the boundaries between any pair of neighbouring prisms, though there is no symmetric way of doing this across the entire triangulation. As a result, the number of tetrahedra in the star of each type of edge cannot be constant throughout the piecewise flat manifold. Any definition of volume $V_\ell$ that depends on the number of tetrahedra in the star of $\ell$ will therefore give different results across the triangulation. The Voronoi volume associated with each vertex is unaffected by this ambiguity however, giving
\begin{equation}
V_v
 = |b| \, \frac{5}{3} \, \mathrm{Area}(a,a,a)
 = \frac{1}{3} \, |b| \, \pi \, r^2 \,  .
\end{equation}
The scalar curvature at each vertex is then
\begin{equation}
R_v
 = \frac{2 \, |b| \, \epsilon_b}{V_v}
 = \frac{2}{r^2} \, ,
\end{equation}
which reduces to the two dimensional computation for the scalar curvature at a vertex of an icosahedron, and matches the smooth manifold value exactly. The volumes $V_{v|b}$ and $V_{v|c}$ are conveniently given by half of the vertex volume due to the orthogonality of $b$ and each icosahedron. The rest of the curvatures are shown in table \ref{tab:Cyl}, and match perfectly with the smooth manifold.

% The resulting triangulation contains three different types of edges. Edges contained in the icosahedra $a$ have their lengths defined by the radius of the $3$-cylinder, those orthogonal to the icosahedra $b$ depend on the choice of distance between successive icosahedra, and the lengths of the diagonal edges $c$ are given by the other two from Pythagoras' theorem. The deficit angles can be seen to vanish for both edges $a$ and $c$, with the deficit angles for $b$ inherited from the vertex deficit angles on the icosahedra, $\epsilon_b = \pi/3$. The volume $V_v$ dual to each vertex is easily seen to be given by the area dual to an icosahedral vertex times the length of the edge $b$. Due to the orthogonality of $b$ with the icosahedra, the volume $V_{|b} = \frac{1}{2} V_v$. The volume for $V_{|c}$ can also be taken as half of the vertex volume.

\begin{table}[h]
\centering
\begin{tabular}{|c||c|c||c|c|c||c|}
\hline
Edge & $|\ell|$ & $\epsilon$ & $K_\ell$ & $\tens{Rc}_\ell$  \\
\hline
$a$ & $2 \sqrt{\frac{\pi}{5 \sqrt{3}}} \, r$ & $0$ & $0$ & $\frac{1}{r^2}$ \\
$b$ & $|b|$ & $\pi/3$ & $ \frac{1}{r^2}$ & $0$ \\
$c$ & $\sqrt{|a|^2 + |b|^2}$ & 0 &
 $\frac{1}{r^2} \cos^2 \theta$ & $\frac{1}{r^2} \sin^2 \theta$ \\
\hline
\end{tabular}
\caption{Edge-lengths and deficit angles for the $3$-cylinder triangulation, with the sectional and Ricci curvature analogues for each of the three edge types.}
\label{tab:Cyl}
\end{table}

From the Ricci curvatures, the evolution under Ricci flow reduces the radius of the spherical part in the same way as Ricci flow for a $2$-sphere. The edges $b$ are left unchanged, and the component of $c$ tangent to the spherical part changes in line with $a$. This ensures that $a$ and $b$ remain orthogonal throughout the evolution, and the characteristics of the triangulation remain unchanged.

\subsection{Gowdy}
\label{sec:Gowdy}

The Gowdy manifold originates from a cosmological model due to Gowdy \cite{Gowdy}, with a propagating gravitational wave. A three-dimensional space-like slice of this model was used in \cite{CarfIsenGowdy} to show convergence to flat space of the Ricci flow of metrics with indefinite Ricci curvature. The line element in coordinates $(x, y, \theta)$ is given as
\begin{equation}
\mathrm{d} s^2
 = e^{f \, W} \mathrm{d} x^2
 + e^{-f \, W} \mathrm{d} y^2
 + e^{2 a} \mathrm{d} \theta^2 .
\label{Gowdy}
\end{equation}
This manifold is not isotropic, but is homogeneous in each $xy$-plane. As with \cite{CarfIsenGowdy}, a compact without boundary submanifold (with a $T^3$ topology and $\theta$ ranging from $0$ to $2 \pi$) will be considered here. The variables in (\ref{Gowdy}) are also chosen as
\begin{equation}
f = 1, \quad W = 0.1 \sin \theta, \quad a = 0,
\end{equation}
representing a simple sinusoidal plane wave.

Two types of triangulations are used. For the first, the $x$ and $y$ coordinates range over the same values throughout $\theta$, leading to a \emph{cubic} triangulation. The second is `skewed' with respect to the first, with the boundary coordinates of each $x y$-plane changing linearly with $\theta$. This removes any right angles, giving \emph{isosceles} tetrahedra and a strongly Delaunay triangulation. In each case the submanifold is divided into blocks along the $\theta$ direction, with each block then subdivided into six tetrahedra, see figure \ref{fig:GowdyTet}. In order to test for convergence, divisions into 6, 12 and 24 blocks are used, with table \ref{tab:GDefAng} showing a decrease in the size of the deficit angles as the number of blocks increases.

\begin{figure}[h!]
\begin{center}
\includegraphics[scale=0.25]{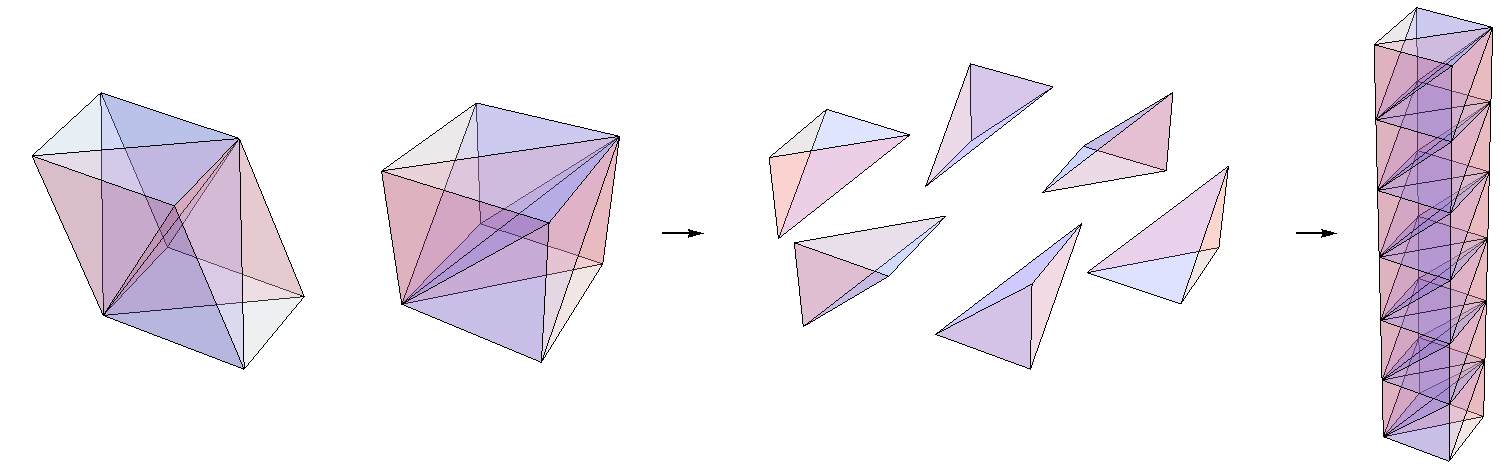}
\end{center}
\vspace{-0.5cm}
\caption{An isosceles and cubic block, with the separation of the cubic block into tetrahedra, and combination of 6 such blocks.}
\label{fig:GowdyTet}
\end{figure}

\begin{table}[h!]
\centering
\begin{tabular}{|c||c|c|c||c|c|c|}
\hline
Blocks &
 Cubic: Max $\epsilon$ & Mean $\epsilon$ & Std. Dev. $\epsilon$ &
  Isos: Max $\epsilon$ & Mean $\epsilon$ & Std. Dev. $\epsilon$ \\
\hline
$6$ &
 2.42 & 0.474 & 0.932 &
 2.67 & 0.669 & 0.828 \\
$12$ &
 0.768 & 0.135 & 0.244 &
 0.839 & 0.185 & 0.215 \\
$24$ &
 0.196 & 0.0346 & 0.0617 &
 0.214 & 0.0472 & 0.0542 \\
\hline
\end{tabular}
\caption{The maximum, mean and standard deviation of the deficit angles in degrees, over all of the edges in each of the three resolutions of each type of Gowdy triangulation.}
\label{tab:GDefAng}
\end{table}

\begin{note}
For computations each $k$-simplex is defined by the set of vertices in its closure. However, since opposite sides of the manifolds are identified to form a $T^3$ topology, ambiguities in these definitions can arise. To avoid this, a covering space with three duplications in each of the $x$ and $y$ directions is used to define the simplices, though only the properties of a single set of vertices or edges require computing.
\end{note}

For the triangulations used here, there is little difference between the Voronoi and barycentric tessellations. In order to show both however, the results for the barycentric are given for the cubic and the Voronoi for the isosceles triangulations. The latter is particularly suited to the Voronoi decomposition since it gives a strongly Delaunay triangulation. Graphs for the scalar curvatures with respect to the $\theta$ coordinate are given in figure \ref{fig:GSca}, showing a clear convergence to the smooth manifold scalar curvature.

\begin{figure}[h!]
\begin{center}
\includegraphics[scale=0.5]{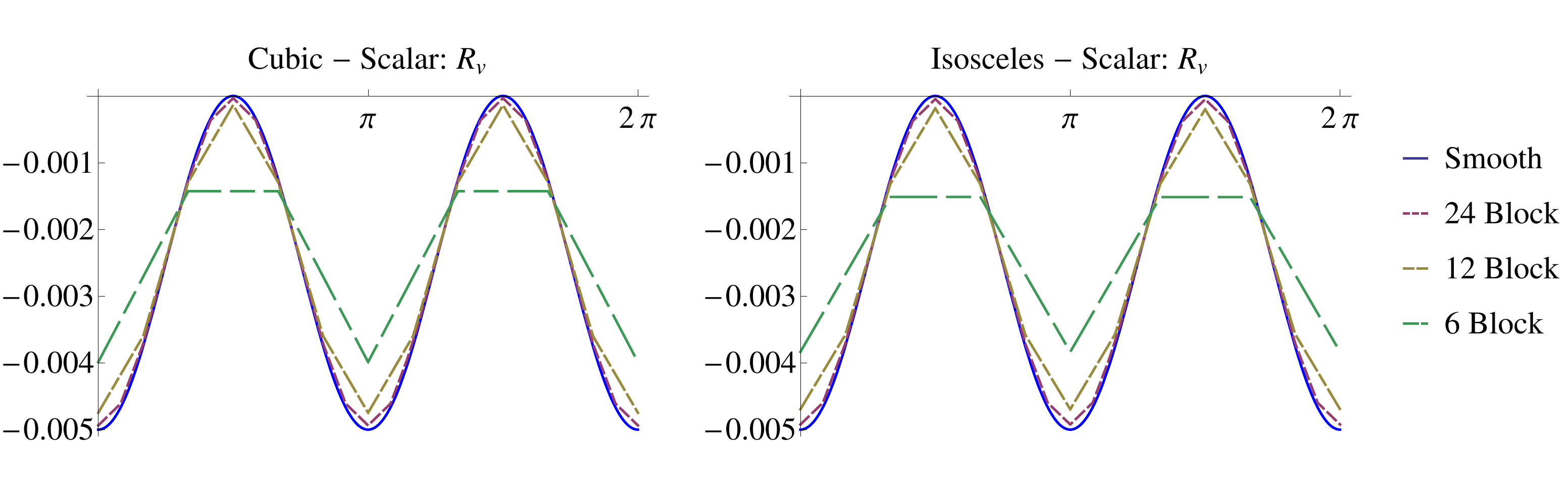}
\end{center}
\vspace{-1cm}
\caption{Graphs of the scalar curvature as a function of $\theta$ for the cubic triangulation on the left and isosceles on the right.}
\label{fig:GSca}
\end{figure}

Due to the numbers of edges, the individual values for the scalar and Ricci curvature computations for each edge are not shown directly. Instead, table \ref{tab:GErr} gives the percentage errors and their standard deviation for each of type of curvature and triangulation. The definition for the percentage error of the sectional curvature is shown below
\begin{equation}
\% \ \textrm{Error}(K_\ell)
 := 100 \, \sum_\ell \left|1 - \frac{K_\ell}{\widetilde K_{\ell_M}}\right| .
\end{equation}
These errors decrease with the increased number of blocks, along with their standard deviations. For the $24$-block triangulations in particular, the errors for the scalar, sectional and Ricci curvatures are all on the order of $1 \%$ for both types of triangulations.

\begin{table}[h!]
\centering
\begin{tabular}{|c||c|c||c|c||c|c|}
\hline
Cubic &
 $R$: \ Error & Std. Dev. &
 $K_\ell$: \ Error & Std. Dev. &
 $\tens{Rc}_\ell$: \ Error & Std. Dev. \\
\hline
$6$  &
 18.2 $\%$ & 17.3 $\%$ &
 12.9 $\%$ & 12.0 $\%$ &
 14.2 $\%$ & 12.0 $\%$ \\
$12$ &
 5.18 $\%$ & 3.03 $\%$ &
 3.36 $\%$ & 3.15 $\%$ &
 3.89 $\%$ & 3.14 $\%$ \\
$24$ &
 1.28 $\%$ & 0.730 $\%$ &
 0.846 $\%$ & 0.800 $\%$ &
 1.01 $\%$ & 0.777 $\%$ \\
\hline
\hline
Isos. &
 $R$: \ Error & Std. Dev. &
 $K_\ell$: \ Error & Std. Dev. &
 $\tens{Rc}_\ell$: \ Error & Std. Dev. \\
\hline
$6$  &
 22.6 $\%$ & 18.6 $\%$ &
 11.0 $\%$ & 9.15 $\%$ &
 10.7 $\%$ & 9.71 $\%$ \\
$12$ &
 6.62 $\%$ & 3.34 $\%$ &
 2.88 $\%$ & 2.68 $\%$ &
 3.01 $\%$ & 2.64 $\%$ \\
$24$ &
 1.64 $\%$ & 0.870 $\%$ &
 1.30 $\%$ & 1.18 $\%$ &
 1.35 $\%$ & 1.12 $\%$ \\
\hline
\end{tabular}
\caption{The percentage error and its standard deviation for the scalar, sectional and Ricci curvatures, with respect to the smooth manifold values.}
\label{tab:GErr}
\end{table}

In order to give a visual representation of the results, graphs are shown in figure \ref{fig:G} for a representative edge from each block, the diagonal in the $x \theta$-face, as a function of $\theta$. The top set of graphs show the lack of correlation between the single-edge sectional curvature (\ref{singlehK}) and the smooth curves, while the middle set in contrast shows the convergence of $K_\ell$ to the same curves, with very close approximations for the higher resolution triangulations. The lower set of graphs also show the convergence of $\tens{Rc}_\ell$ to the smooth curves, and the close approximation of the $24$-block triangulations in particular.

\begin{figure}[h!]
\begin{center}
\includegraphics[scale=0.5]{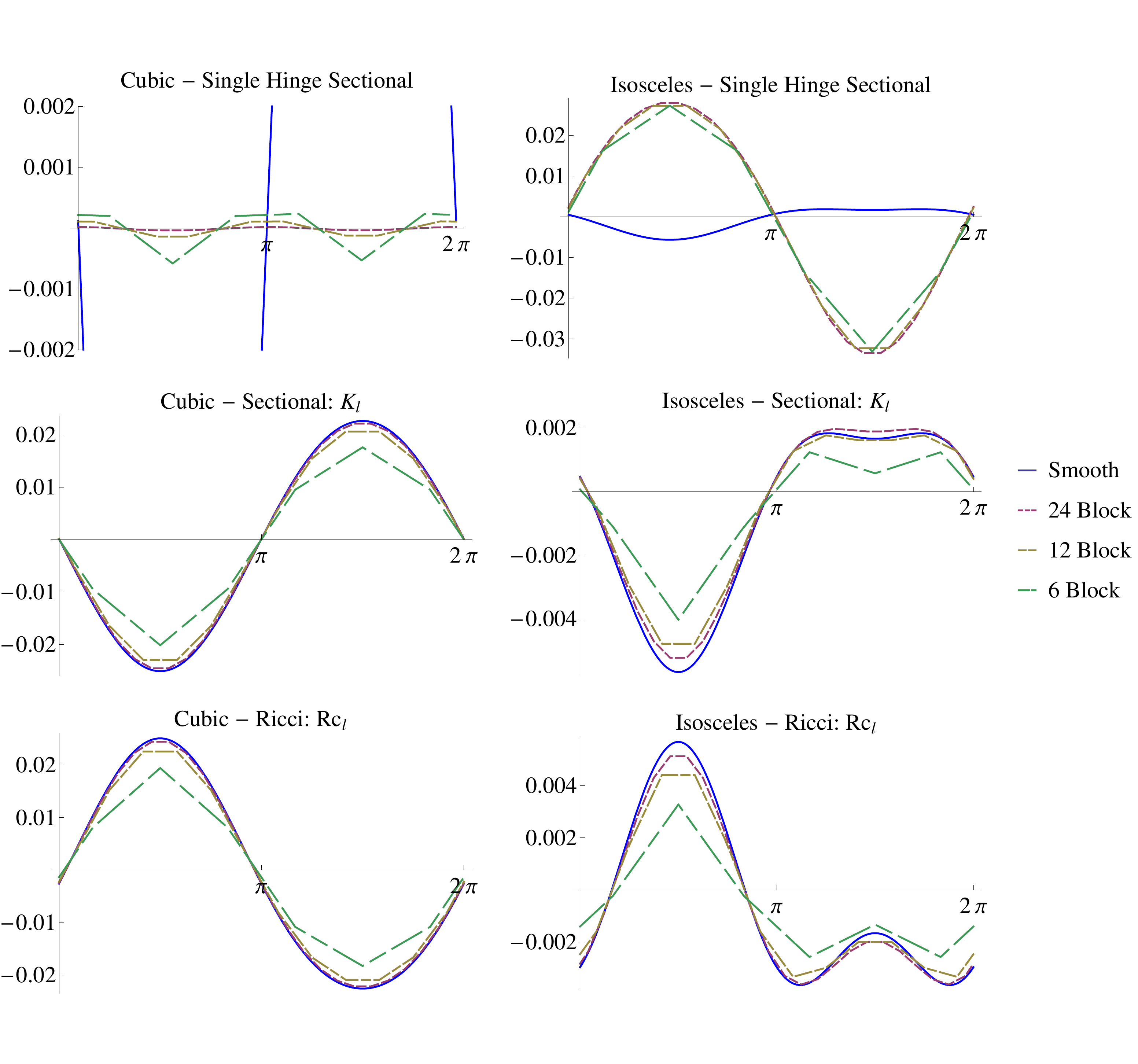}
\end{center}
\vspace{-1.5cm}
\caption{Graphs for the diagonal edge in the $x \theta$-face for the cubic triangulation on the left and isosceles on the right. The top set of graphs show the lack of correlation of the single-hinge sectional curvature with the smooth curvature, with the middle and lower sets showing the convergence of $K_\ell$ and $\tens{Rc}_\ell$}
\label{fig:G}
\end{figure}

The normalized Ricci flow was computed for each triangulation, and gives equivalent short-time behaviour to the smooth Ricci flow, with the length of time increasing for higher resolution triangulations. More work is required to give convergence to a flat metric however.

\begin{remark}
The errors for the single-edge sectional curvatures are on the order of $100 \%$, with no decrease for increased resolutions of the triangulations. The more general Ricci curvature suggested in proposition \ref{prop:RcSn}, using the volumes $V_\ell$, gives similar error values to $\tens{Rc}_\ell$ for the cubic triangulations, but errors on the order of $20 \%$ for the isosceles triangulations. Such a triangulation dependence implies that at the very least, the choice of the volume $V_\ell$ is not an appropriate choice for computing the Ricci curvature.
\end{remark}

\subsection{Nil-3}
\label{sec:Nil3}

The Nil-$3$ manifold is one of the eight Thurston model geometries, with a nilpotent group of diffeomorphisms given by the Heisenberg group \cite{Thurston}. Ricci flow of Nil-$3$ manifolds has been studied in \cite{GIKlinear} and \cite{KnModel}, among others. In Cartesian-type coordinates $(x, y, z)$, the metric can be given as
\begin{equation}
\mathrm{d} s^2
 = \mathrm{d} x^2
 + \mathrm{d} y^2
 + (\mathrm{d} z - x \, \mathrm{d} y)^2 .
\label{Nil3}
\end{equation}
The manifold can be decomposed into compact without boundary submanifolds by quotienting the Heisenberg group by the discrete Heisenberg group. These submanifolds have an almost $T^3$ topology, with a periodic range for both the $y$ and $z$ coordinates, and a `twisted' correspondence between the bounding $y z$-faces, as shown in figure \ref{fig:N3T3}. Here, the $x$-coordinate is chosen to range between $0$ and $1$, with the $y$ and $z$-coordinates having the same range as each other, so that the $y z$-faces are square.

\begin{figure}[h!]
\begin{center}
\includegraphics[scale=0.3]{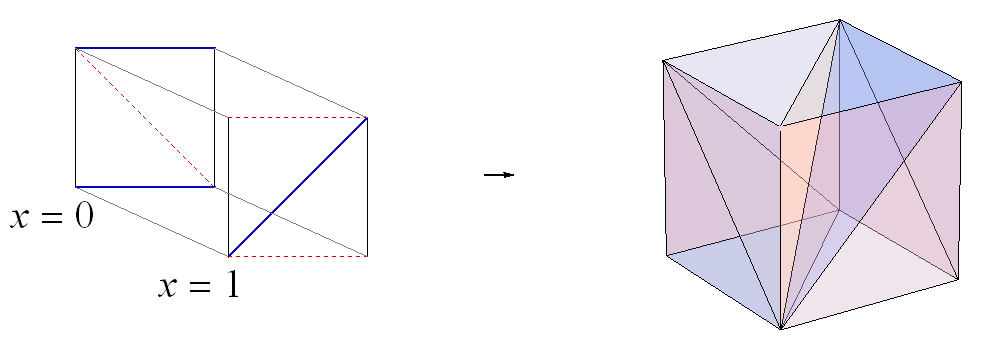}
\end{center}
\vspace{-0.5cm}
\caption{The correspondence of the $y z$ faces at $x = 0$ and $x = 1$, and the first block in each triangulation with the same correspondence between the front and back faces.}
\label{fig:N3T3}
\end{figure}

Triangulations are formed in a similar way to the cubic triangulations of the Gowdy model, divided into $6$, $12$ and $24$ blocks along the $x$-direction, with each block subdivided into $6$ tetrahedra. The first block has a slightly different structure to the rest, in order to account for the twisting correspondence of the $y z$ bounding faces, and is shown in figure \ref{fig:N3T3}. As with the Gowdy computations, a covering space of nine such submanifolds was used to define the simplicial graph, with the equations for only a single set of vertices and edges requiring solutions. The decrease in the deficit angles for the increased number of blocks can be seen in table \ref{tab:N3DefAng}.

\begin{table}[h!]
\centering
\begin{tabular}{|c||c|c|c||c|c|c|}
\hline
Blocks &
 Max $\epsilon$ & Mean $\epsilon$ & Std. Dev. $\epsilon$ \\
\hline
$6$ &
 1.23 & 0.363 & 0.383 \\
$12$ &
 0.309 & 0.0882 & 0.0956 \\
$24$ &
 0.0773 & 0.0217 & 0.0235 \\
\hline
\end{tabular}
\caption{The maximum, mean and standard deviation of the deficit angles in degrees, over all of the edges in each of the three resolutions the Nil-$3$ triangulation.}
\label{tab:N3DefAng}
\end{table}

Due to the nature of the topology, finding Delaunay triangulations is not straight forward. Those that have been found are also unstable, with evolution under Ricci flow leading to non-Delaunay triangulations in very short time scales. A barycentric tessellation was therefore chosen to give the dual volumes. The percentage errors for each of the three types of curvatures are shown in table \ref{tab:N3Err}, with a clear decrease in the errors for higher resolutions of the curvature.

\begin{table}[h!]
\centering
\begin{tabular}{|c||c|c||c|c||c|c|}
\hline
Blocks &
 $R$: \ Error & Std. Dev. &
 $K_\ell$: \ Error & Std. Dev. &
 $\tens{Rc}_\ell$: \ Error & Std. Dev. \\
\hline
$6$  &
 0.544 $\%$ & 0.425 $\%$ &
 6.20 $\%$ & 7.19 $\%$ &
 4.82 $\%$ & 5.70 $\%$ \\
$12$ &
 0.162 $\%$ & 0.166 $\%$ &
 4.19 $\%$ & 6.39 $\%$ &
 3.27 $\%$ & 5.00 $\%$ \\
$24$ &
 0.0432 $\%$ & 0.0642 $\%$ &
 3.57 $\%$ & 6.02 $\%$ &
 2.79 $\%$ & 4.70 $\%$ \\
\hline
\end{tabular}
\caption{The percentage error and its standard deviation for the scalar, sectional and Ricci curvatures, with respect to the smooth manifold values.}
\label{tab:N3Err}
\end{table}

A graph of the scalar curvature for each triangulation as a function of $x$ is shown in figure \ref{fig:N3}, and can be seen to converge to the smooth value. The peaks in the early part of the graph correspond to the differently triangulated block. Graphs of the sectional and Ricci curvatures for the $x$-coordinate edge in each block are also shown in figure \ref{fig:N3} as a functions of $x$. The single-edge sectional curvature (\ref{singlehK}) can again be seen to converge to a different curve from the smooth sectional curvature, while the multi-hinge expression $K_\ell$ and the Ricci curvature $\tens{Rc}_\ell$ give very close approximations to their respective smooth curves.

\begin{figure}[h!]
\begin{center}
\includegraphics[scale=0.5]{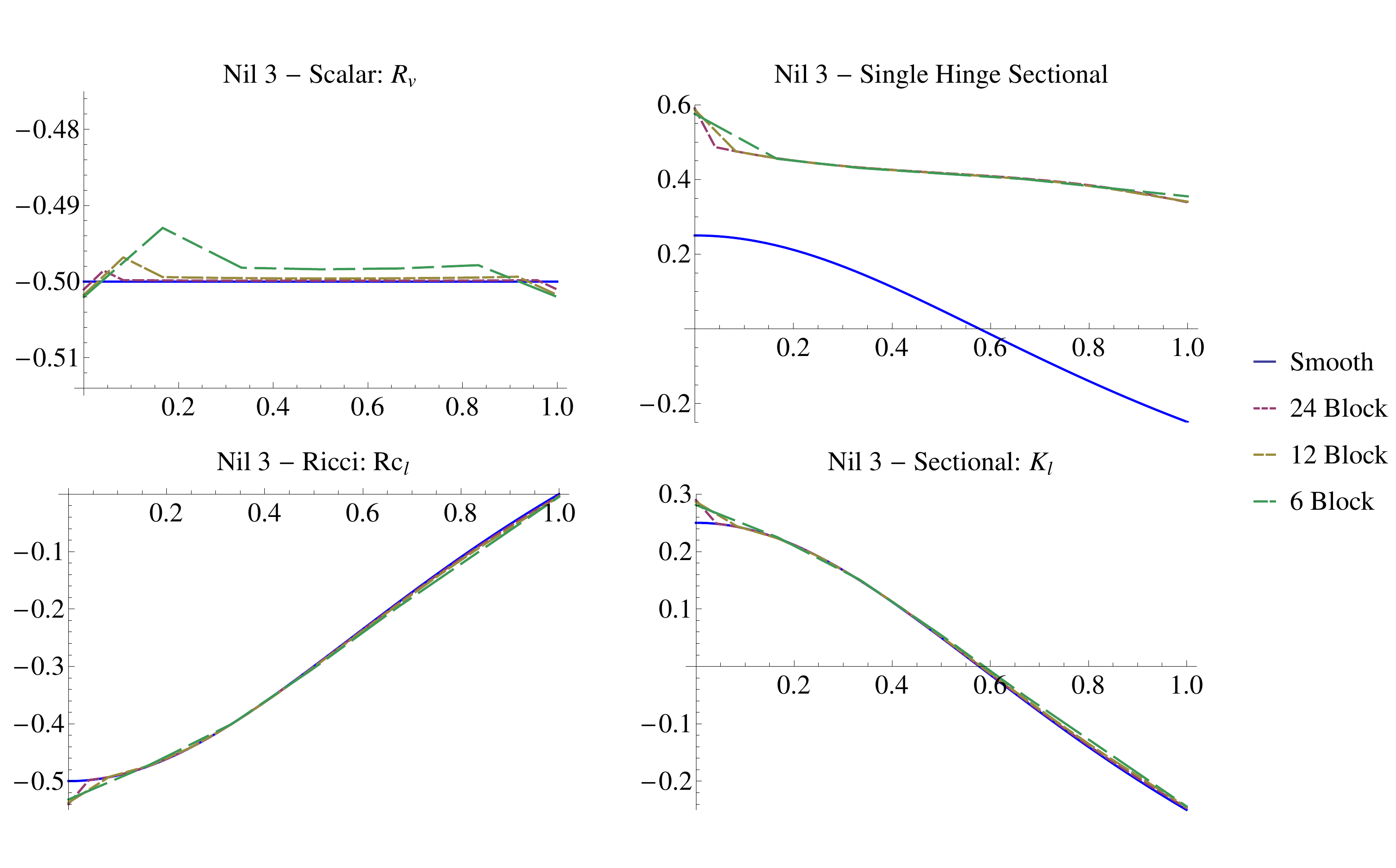}
\end{center}
\vspace{-1cm}
\caption{A graph of the scalar curvature $R_v$ on the top left, and graphs of the single-hinge sectional curvature, $K_\ell$ and $\tens{Rc}_\ell$ for the $x$-coordinate edge in each block, all as functions of $x$.}
\label{fig:N3}
\end{figure}

The results for the piecewise flat simplicial Ricci flow of a single block triangulation are shown in figure \ref{fig:N3RF}, with only six tetrahedra and seven edges. For such a small triangulation, the flow of the edge-lengths is remarkably close to that of their corresponding geodesic segments under smooth Ricci flow. This correspondence also improves for higher resolution triangulations.

\begin{figure}[h!]
\begin{center}
\includegraphics[scale=0.5]{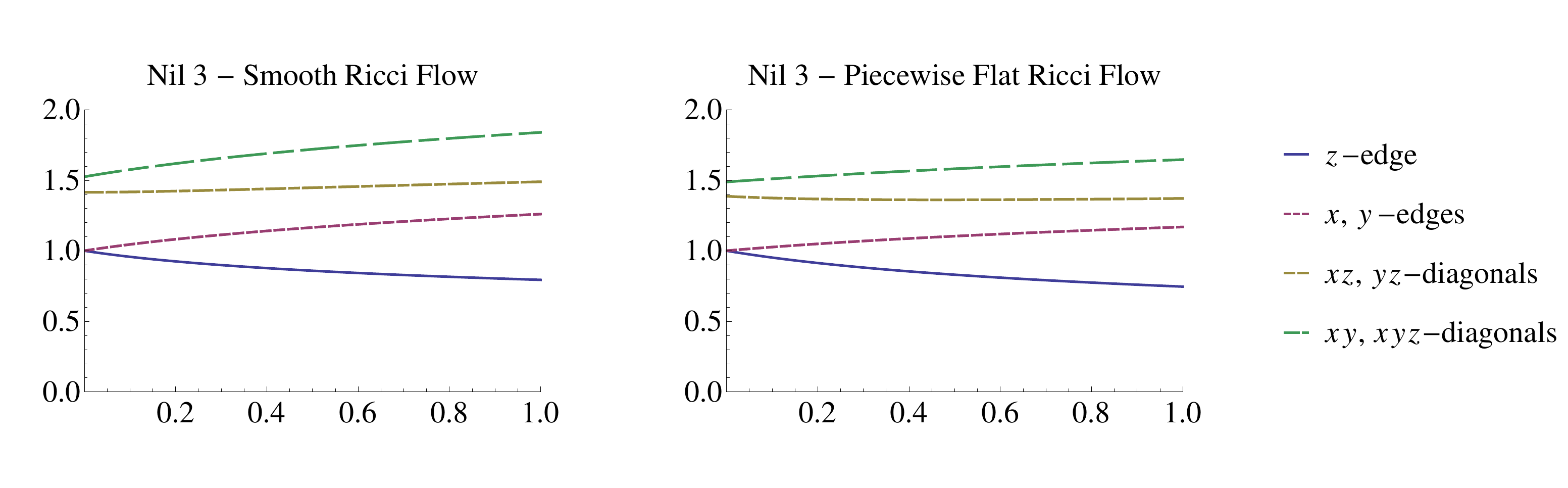}
\end{center}
\vspace{-1cm}
\caption{The lengths of the smooth geodesic segments and the corresponding piecewise flat edges with respect to time for smooth and piecewise flat Ricci flow respectively.}
\label{fig:N3RF}
\end{figure}

\section{Conclusion}

Piecewise flat scalar, sectional and Ricci curvatures and Ricci flow have been constructed in three-dimensions, depending combinatorially on the simplicial structure and also on a choice of dual tessellation. The results of computations have also shown these to converge to their smooth values for triangulations of the $3$-sphere, $3$-cylinder, Gowdy and Nil-$3$ manifolds.

While the computations in section \ref{sec:Comp} give strong support for the constructions, work is ongoing to compute these curvatures on manifolds with less symmetry, such as a perturbed $3$-sphere and neck-pinched dumbbell model. More detailed investigations are also being performed on the piecewise flat Ricci flow, and its dependence on triangulation properties in particular.

% iop
%\ack

% svj
%\begin{acknowledgements}
%\end{acknowledgements}

\appendix

\section{Dual Tessellations}
\label{sec:Duals}

Definitions are given below for both the Voronoi and barycentric tessellations. Though no specific choice of dual volumes has been specified in the development of the theory, these are the tessellations that have been used for the computations in section \ref{sec:Comp}.

\begin{definition}[Circumcentric dual and Voronoi tessellation]

\

\begin{enumerate}
\item A Voronoi tessellation of a piecewise flat manifold $S^n$ associates to each vertex $v$ those points which are closer to $v$ than to any other vertex.

\item For a Delaunay triangulated piecewise flat manifold, the Voronoi $n$-volume $V_v^{(c)}$ associated with a vertex $v$ is the $n$-polytope formed by the circumcenters of the $n$-simplices in the star of $v$.

\item For a Delaunay triangulation $S^n$, for each co-dimension-$1$ simplex $\sigma^{n-1}$ there exists a dual edge $\lambda_s$ formed by joining the circumcenters of the two $n$-simplices in the star of $\sigma^{n-1}$. The extension of $\lambda_s$ is orthogonal to $s$ and passes through its circumcenter.

\item Circumcentric hybrid cells can also be defined for each $k$-simplex $\sigma^k$ as the $n$-polytopes formed by the vertices in the closure of $\sigma^k$ and the circumcenters of the $n$-simplices in its star.
\end{enumerate}
\end{definition}

The circumcentric dual is particularly useful due to its orthogonality properties, giving straight-forward computations for the volumes dual to each vertex. Difficulties with Voronoi tessellations can arise however since the circumcenters of simplices do not always lie in their interior. This can lead to inverted dual edges $\lambda_s$ where the triangulation of $S^n$ will no longer be Delaunay, see for example \cite{HambWill} or \cite{DyerPhd}. When a triangulation is no longer Delaunay, the third point in definition \ref{def:V_v} for the volume $V_v$ will not be satisfied everywhere, and in extreme situations dual and hybrid volumes may vanish.

\begin{definition}[Barycentric tessellation]

\

\begin{enumerate}
\item The barycentric tessellation gives a subdivision of each $k$-simplex $\sigma^k$ into $k + 1$ parts of equal $k$-volumes, each associated with a vertex in the closure of $\sigma^k$.

\item The barycenter of an edge is defined as the midpoint of the edge. The barycenter of a $k$-simplex $\sigma^k$ is then given by the intersection point of the median lines joining each vertex in the closure of $\sigma^k$ to the barycenter of the opposite $(k - 1)$-simplex.

\item The barycentric $n$-volume $V_v^{(b)}$ associated with a vertex $v$ is the $n$-polytope formed by the barycenters of all of the simplices in the star of $v$.

\item Barycentric hybrid cells are defined for each $k$-simplex $\sigma^k$ as the $n$-polytopes formed by the vertices in the closure of $\sigma^k$ and the barycenters of the simplices its star.
\end{enumerate}
\end{definition}

Barycentric tessellations do not have the same issues as the Voronoi, since the barycenter of a $k$-simplex will always lie in its interior. This gives a broader range of applicability, since the Delaunay condition is not required, a condition that becomes more and more restrictive as the dimension is increased. The barycentric dual volumes are also easy to compute once the volumes of the $n$-simplices in the star of a vertex are known.

However, while the Voronoi dual volumes depend mostly on the distribution of vertices, the barycentric dual volumes depend specifically on the volumes of the $n$-simplices. This makes the barycentric tessellation less suitable for situations like the regular triangulation of the $3$-cylinder in section \ref{sec:3Cyl}, where vertices that should all be symmetric can have different numbers of equal volume tetrahedra in their stars.

\section{Possible extension of Ricci curvature to higher dimensions}
\label{sec:RcSn}

The Ricci curvature expression in section \ref{sec:Ric} does not follow the same approach as both the scalar and sectional curvatures. This approach can be used however, if some appopriate $n$-volume $H_\ell$ associated with each edge $\ell$ can be defined. Using the relation for the smooth Ricci curvature along a vector in terms of sectional curvatures (\ref{RcK}), a similar procedure to the scalar curvature construction for theorem \ref{thm:R} can be followed, using $2$-surfaces parallel with $\ell$ for each hinge $h$ instead of those orthogonal to $h$.

% Need to define $H_\ell$ equivalent to $V_\ell$ in $3$ dimensions

\begin{proposition}[Possible piecewise flat Ricci curvature for $S^n$]
\label{prop:RcSn}
The average Ricci curvature in the direction of an edge $\ell \subset S^n$, over a region $H_\ell$, is given by the expression
\begin{equation}
\widetilde{\tens{Rc}}_{H_\ell}
 = \frac{\sum_h |h_{|H_\ell}| \, \sin^2 \theta_h \, \epsilon_h}{H_\ell} ,
\label{RcH}
\end{equation}
for each hinge $h$ intersecting $H_\ell$, with $h_{|H_\ell} = h \cap H_\ell$, $\theta_h$ representing the angle between $\ell$ and $h$, and $\epsilon_h$ the deficit angle at $h$.
\end{proposition}

\begin{proof}
The region $H_\ell$ can be decomposed into regions $D_h$, each enclosing the restriction of a hinge $h$ to $H_\ell$,
\begin{equation}
h_{|H_\ell} \subset D_h, \qquad
 h_i \cap D_{h_j} = \emptyset \quad
 \forall \ i \neq j, \qquad
 |H_\ell| = \sum_{h | h \cap H_\ell} |D_h|.
\end{equation}
Over each region $D_h$, the integral of the Ricci curvature parallel to $\ell$ can be given in terms of the integrals of sectional curvatures for a complete orthogonal set of almost-flat $2$-surfaces $P_i$ parallel to $\ell$. The choice of surfaces $P_i$ at each point $p \in h$ can be made so that one is orthogonal to the line joining $p$ orthogonally to $\ell$. It can easily be seen that this surface will be $P_h^{\pi - \theta_h}$, with the remaining span of surfaces parallel to $h$. The only sectional curvature contribution then comes from ${\cal K}_h^{\pi - \theta_h}$,
\begin{equation}
\int_{D_h} \mathrm{Rc}(\ell) \ \mathrm{d} V^n
 = \int_{D_h} \left[
   \sum_{i} K(P_i)
   \right] \mathrm{d} V^n
 = \int_{h} {\cal K}_h^{\pi - \theta_h} \ \mathrm{d} V^{(n-2)}
 \simeq \int_{h} \ \mathrm{d} V^{(n-2)} \sin \theta_h \epsilon_h .
\end{equation}
The integral in the final term above evaluates to $|h_{|H_\ell}| \sin \theta$, using similar arguments as the proof of theorem \ref{thm:Kl}. The integral of the sectional curvature over all of $H_\ell$ is the sum of the integrals over each of $D_h$,
\begin{equation}
\int_{V_v} R \ \mathrm{d} V^n
 = \sum_{h | h \cap H_\ell}
   \int_{D_h} \mathrm{Rc}(\ell) \ \mathrm{d} V^n
 = \sum_{h | h \cap H_\ell} |h_{|H_\ell}| \sin^2 \theta_h \epsilon_h ,
\end{equation}
assuming small deficit angles. The average curvature is then found by dividing the integral by the volume of $H_\ell$.

\end{proof}

In three dimensions, if the volume $H_\ell := V_\ell$, then (\ref{RcH}) above is equivalent to (\ref{RcM3}) if and only if the average scalar curvature can be constructed directly over $V_\ell$ as
\begin{equation}
\widetilde{R}_{V_\ell}
 = \frac{2}{|V_\ell|}\sum_{h | h \cap H_\ell} |h_{|H_\ell}|\, \epsilon_h .
\end{equation}
Unfortunately, while $V_\ell$ gives an appropriate volume for constructing the sectional curvature, the scalar and Ricci curvatures are quite sensitive to the boundaries of $V_\ell$ within each vertex volume $V_v$. Hinges which are close to this boundary make an angle close to $\frac{\pi}{2}$ with the edge $\ell$ by construction. While these hinges are given a weighting of $\cos^2 \theta$ for the sectional curvature, which goes to zero as the boundary is approached, the scalar and Ricci curvatures are weighted by $1$ and $\sin^2 \theta$ respectively. A slight variation of a hinge to the other side of the boundary will therefore have little effect on the sectional curvature, but can have a large effect on the scalar or Ricci curvatures evaluated over $V_\ell$.

Searches for a more appropriate definition of the volume $H_\ell$ have not yet been successful, but work continues in this direction.

% iop
%\section*{References}

\bibliography{SimpGeom}

\begin{thebibliography}{10}

\bibitem{Regge}
T~Regge.
\newblock General relativity without coordinates.
\newblock {\em Nuovo Cimento}, 19, 558, 1961.

\bibitem{Forman}
R~Forman.
\newblock {Bochner's} method for cell complexes and combinatorial {Ricci}
  curvature.
\newblock {\em Discrete Comput. Geom.}, 29: 353-374, 2003.

\bibitem{CoopRiv}
D~Cooper and I~Rivin.
\newblock Combinatorial scalar curvature and rigidity of ball packings.
\newblock {\em Math. Res. Lett.}, {\bf 3}, no. 1, 51-60, 1996.

\bibitem{Glick}
D~Glickenstein.
\newblock Discrete conformal variations and scalar curvature on piecewise flat
  two- and three-dimensional manifolds.
\newblock {\em J. Differential Geom.}, {\bf 87} 201-237, 2011.

\bibitem{LuoYam}
F~Luo.
\newblock Combinatorial {Yamabe} flow on surfaces.
\newblock {\em Commun. Contemp. Math.}, {\bf 06} 765, 2004.

\bibitem{GlickYam}
D~Glickenstein.
\newblock A combinatorial {Yamabe} flow in three dimensions.
\newblock {\em Topology}, 44 791-808, 2005.

\bibitem{GlickYamMax}
D~Glickenstein.
\newblock A maximum principle for combinatorial {Yamabe} flow.
\newblock {\em Topology}, 44 809-825, 2005.

\bibitem{CombRF}
B~Chow and F~Luo.
\newblock Combinatorial {Ricci} flows on surfaces.
\newblock {\em J. Differential Geom.}, {\bf 63} 97-129, 2003.

\bibitem{HiraniPhD}
Anil~N Hirani.
\newblock {\em Discrete Exterior Calculus}.
\newblock PhD thesis, California Institute of Technology, 2003.

\bibitem{DEC2}
M~Desbrun, A~N Hirani, M~Leok, and J~E Marsden.
\newblock Discrete exterior calculus.
\newblock {\em arXiv:math/0508341v2}, 2005.

\bibitem{HambWill}
H~W Hamber and R~M Williams.
\newblock Higher derivative quantum gravity on a simplicial lattice.
\newblock {\em Nucl. Phys. B}, 248 392-414, 1984.

\bibitem{Mbbp}
W~A Miller.
\newblock The geometrodynamic content of the {Regge} equations as illuminated
  by the boundary of a boundary principle.
\newblock {\em Found. Phys.}, 16, 143-169, 1989.

\bibitem{SRF}
W~A Miller, J~R McDonald, P~M Alsing, D~X Gu, and S-T Yau.
\newblock Simplicial {Ricci} flow.
\newblock {\em Commun. Math. Phys.}, 329, 579-608, 2014.

\bibitem{TraceK}
R~Conboye, W~A Miller, and S~Ray.
\newblock Distributed mean curvature on a discrete manifold for {Regge}
  calculus.
\newblock {\em Class. Quantum Grav.}, {\bf 32} 185009, 2015.

\bibitem{BobSpring}
A~I Bobenko and B~A Springborn.
\newblock A discrete {Laplace}-{Beltrami} operator for simplicial surfaces.
\newblock {\em Discrete Comput. Geom.}, {\bf 38} 740-756, 2007.

\bibitem{MNeckPinch1}
P~M Alsing, W~A Miller, M~Corne, X~D Gu, S~Lloyd, S~Ray, and S-T Yau.
\newblock Simplicial {Ricci} flow: an example of a neck pinch singularity in
  3d.
\newblock {\em Geom., Imaging Comp.}, {\bf 1}, no. 3, 303-331, 2014.

\bibitem{MillNand}
P~M Alsing, H~A Blair, M~Corne, G~Jones, W~A Miller, K~Mischaikow, and V~Nanda.
\newblock Topological signatures of singularities in simplicial {Ricci} flow.
\newblock {\em arXiv:1502.02630v1}, 2015.

\bibitem{RovSmo}
C~Rovelli and L~Smolin.
\newblock Discreteness of area and volume in quantum gravity.
\newblock {\em Nucl. Phys. B}, {\bf 442} 593-619, 1995.

\bibitem{NonAss}
A~I Nesterov and L~V Sabinin.
\newblock Nonassociative geometry: Towards discrete structure of spacetime.
\newblock {\em Phys. Rev. D}, {\bf 62} 081501, 2000.

\bibitem{CausSets}
F~Dowker.
\newblock Causal sets as discrete spacetime.
\newblock {\em Contemp. Phys.}, {\bf 47}, no. 1, 1-9, 2006.

\bibitem{RegWill}
T~Regge and R~M Williams.
\newblock Discrete structures in gravity.
\newblock {\em arXiv:gr-qc/0012035}, 2000.

\bibitem{RovVid}
C~Rovelli and F~Vidotto.
\newblock {\em Covariant Loop Quantum Gravity}.
\newblock Cambridge University Press, Cambridge, 2015.

\bibitem{Sch}
R~Schrader.
\newblock Piecewise linear manifolds: {Einstein} metrics and {Ricci} flows.
\newblock {\em arXiv:1508.05520v3}, 2015.

\bibitem{FriedLee}
R~Friedberg and T~D Lee.
\newblock Derivation of {Regge's} action from {Einstein's} theory of general
  relativity.
\newblock {\em Nucl. Phys. B}, 242 145-166, 1984.

\bibitem{Christ}
S~H Christiansen.
\newblock Exact formulas for the approximation of connections and curvature.
\newblock {\em arXiv:1307.3376v2}, 2015.

\bibitem{PiranWill}
T~Piran and R~M Williams.
\newblock Three-plus-one formulation of {Regge} calculus.
\newblock {\em Phys. Rev. D}, 33, 06 1622, 1985.

\bibitem{HilbertAction}
W~A Miller.
\newblock The {Hilbert} action in {Regge} calculus.
\newblock {\em Class. Quantum Grav.}, {\bf 14} L199-L204, 1997.

\bibitem{MMR}
J~R McDonald and W~A Miller.
\newblock A geometric construction of the {Riemann} scalar curvature in {Regge}
  calculus.
\newblock {\em Class. Quantum Grav.}, 25 195017, 2008.

\bibitem{CMS}
J~Cheeger, W~M{\"u}ller, and R~Schrader.
\newblock On the curvature of piecewise flat spaces.
\newblock {\em Commun. Math. Phys.}, 92, 405-454, 1984.

\bibitem{SRT}
P~M Alsing, J~R McDonald, and W~A Miller.
\newblock The simplicial {Ricci} tensor.
\newblock {\em Class. Quantum Grav.}, 28 155007, 2011.

\bibitem{ChowKnopfRF}
Bennet Chow and Dan Knopf.
\newblock {\em The {Ricci} Flow: An Introduction}.
\newblock American Mathematical Society, 2004.

\bibitem{HamRf}
R~Hamilton.
\newblock Three-manifolds with positive ricci curvature.
\newblock {\em J. Differential Geom.}, {\bf 17} 255-306, 1982.

\bibitem{DoubleTet}
D~Champio, D~Glickenstein, and A~Young.
\newblock {Regge's} {Einstein}-{Hilber} functional on the double tetrahedron.
\newblock {\em Differential Geom. Appl.}, {\bf 29} 108-124, 2011.

\bibitem{MNeckPinch2}
W~A Miller, P~A Alsing, M~Corne, and S~Ray.
\newblock Equivalence of simplicial {Ricci} flow and {Hamilton's} {Ricci} flow
  for 3d neckpinch geometries.
\newblock {\em Geom., Imaging Comp.}, {\bf 1}, no. 3, 333-366, 2014.

\bibitem{Gowdy}
R~Gowdy.
\newblock Vacuum spacetimes with two-parameter spacelike isometry groups and
  compact invariant hypersurfaces.
\newblock {\em Ann. Phys.}, {\bf 83} 203-241, 1974.

\bibitem{CarfIsenGowdy}
M~Carfora, J~Isenberg, and M~Jackson.
\newblock Convergence of the {Ricci} flow for metrics with indefinite {Ricci}
  curvature.
\newblock {\em J. Differential Geom.}, {\bf 31} 249-263, 1990.

\bibitem{Thurston}
W~P Thurston.
\newblock {\em Three-dimensional Geometry and Topology}, volume~1.
\newblock Princeton University Press, 1997.

\bibitem{GIKlinear}
C~Guenther, J~Isenberg, and D~Knopf.
\newblock Linear stability of homogeneous {Ricci} solitons.
\newblock {\em Int. Math. Res. Not.}, 96253, 2006.

\bibitem{KnModel}
D~Knopf and K~McLeod.
\newblock Quasi-convergence of model geometries under the {Ricci} flow.
\newblock {\em Comm. Anal. Geom.}, {\bf 9}, no. 4, 879-919, 1999.

\bibitem{DyerPhd}
Ramsey Dyer.
\newblock {\em Self-Delaunay Meshes for Surfaces}.
\newblock PhD thesis, Simon Fraser University.

\end{thebibliography}
\bibliographystyle{unsrt}

\end{document}